\documentclass[10pt]{amsart} 
\usepackage{amscd,amssymb,graphicx,colortbl}

\usepackage{amsfonts}
\usepackage{amsmath}
\usepackage{amsxtra}
\usepackage{latexsym}
\usepackage[mathcal]{eucal}  
\usepackage{enumitem} 
\usepackage{tikz}

\input xy
\xyoption{all}

\usepackage[pdftex,bookmarks,colorlinks,breaklinks]{hyperref}

\newtheorem{thm}{Theorem}[section]
\newtheorem{f}[thm]{Fact} 
\newtheorem{cor}[thm]{Corollary}
\newtheorem{lem}[thm]{Lemma}
\newtheorem{problem}[thm]{Problem}

\newtheorem{prop}[thm]{Proposition}

\theoremstyle{definition}
\newtheorem{defin}[thm]{Definition}

\theoremstyle{remark}
\newtheorem{remark}[thm]{Remark}

\newtheorem{ex}[thm]{Example}

\numberwithin{equation}{section}

\newcommand{\delete}[1]{} 

\def\eps{{\varepsilon}}
\newcommand{\sk}{\vskip 0.2cm}

\newcommand{\ben}{\begin{enumerate}}

\newcommand{\een}{\end{enumerate}}
\newcommand{\bit}{\begin{itemize}}

\newcommand{\eit}{\end{itemize}}

\def\R {{\mathbb R}}
\def\N {{\mathbb N}}
\def\Z {{\mathbb Z}}

\newcommand{\cU}{\mathfrak{U}}

\def\St{{\mathrm St}\,}

\def\Ucal{\mathcal{U}}

\def\rank{\operatorname{rank}} 

\def\LOTS{\operatorname{LOTS}}

\def\Ends{\operatorname{Ends}}

\def\a{\alpha}

\def\t{\tau}

\def\g{\gamma}

\hypersetup{
	linkcolor=blue,
	citecolor=red,
	urlcolor=red
}

\begin{document}
	\raggedbottom   

\title[]
{Intrinsic uniform structure on median algebras} 
 
\author[]{Michael Megrelishvili} 
\address{Department of Mathematics,
Bar-Ilan University, 52900 Ramat-Gan, Israel}
\email{megereli@math.biu.ac.il}
\urladdr{http://www.math.biu.ac.il/$^\sim$megereli}
 
 \thanks{Supported by the Gelbart Research Institute at the Department of Mathematics, Bar-Ilan  University}  
 
 

\date{October, 2026} 
 
\begin{abstract}  
We introduce the \textit{median uniformity} $\mathcal U_{\mathrm m}$, an intrinsic precompact convex uniform structure on a median algebra. It is Hausdorff under natural assumptions, for instance for finite-rank median algebras. In the Hausdorff case, its uniform completion yields the \textit{Minimal Median Compactification} (MMC). The induced topology $\tau_{\mathrm m}$ provides a natural higher-rank analogue of the interval topology on linearly ordered sets and of the shadow topology on rank-one median algebras. 
When the intrinsic topology $\tau_{\mathrm m}$ is Hausdorff and all intervals are compact, the MMC is the unique proper median compactification and is canonically isomorphic to the Roller--Fioravanti compactification.

We apply this uniform framework to continuous actions of a topological group $G$ by median automorphisms. We show that the MMC is a median $G$-compactification. In the finite-rank case, the resulting compact $G$-system is Rosenthal representable and dynamically tame.

\end{abstract}
\maketitle

\setcounter{tocdepth}{1}
\tableofcontents

\section{Introduction} \label{s:intro}

Abstract median algebra theory provides a geometric bridge between tree-like structures,
linearly ordered sets, and non-positively curved CAT(0) cube complexes, while also
connecting naturally with general convexity theory, topology, and metric geometry.

By recent results of \cite{Me-CircTop}, for every linear order there exists a canonical precompact convex \textit{interval uniformity} $\mu_{\leq}$ compatible with the interval topology $\lambda_{\leq}$, whose uniform  completion yields the classical Dedekind compactification. 
Every linearly ordered set carries a natural median algebra structure of rank 1. Our aim is to generalize this interval uniformity to median algebras. 

The theory of convex structures developed by van de Vel~\cite{Vel-book} and many
other authors provides a rich framework for topologies and uniformities compatible
with a given convexity.

The median uniformity $\mathcal U_{\mathrm m}$ introduced in the present paper is constructed in a median-specific and purely algebraic way.  
Motivated by the uniformity $\mu_{\leq}$ on linearly ordered spaces, we introduce an \textit{intrinsic median uniformity} $\mathcal U_{\mathrm m}$ on median algebras $(X,m)$, generated by two-member branch covers $\{B^u_v, B^v_u\}$ associated with nontrivial chain intervals $[u,v]$ of $X$ (cf. Definition \ref{d:StarMedianUniformity}).  

The uniformity \(\mathcal U_{\mathrm m}\) is precompact because it is generated
by finite two-member covers. It is convex because every branch
associated with a chain interval is convex. 
Under natural assumptions, for instance, when $X$ has finite
rank $\mathcal U_{\mathrm m}$ is Hausdorff. 
 In the Hausdorff case, the uniform completion \(\widehat X\) of $(X,\mathcal U_{\mathrm m})$ serves as the \textit{Minimal Median Compactification} (MMC) of the topological median algebra $(X,\tau_{\mathrm m})$, where \(\tau_{\mathrm m}\) is the topology of $\Ucal_{\mathrm m}$. We call \(\tau_{\mathrm m}\) the \textit{intrinsic m-topology} of $(X,m)$. 
Equivalently, it can be described using the closed subbase consisting of the
corresponding shadows \(S^u_v\) and \(S^v_u\) coming from nontrivial chain intervals \([u,v]\). 
For every separable metrizable \((X,\tau_{\mathrm m})\) the MMC \(\widehat X\) is a metrizable compactum 
(Theorem \ref{t:metrizableMMC}). 

The topology \(\tau_{\mathrm m}\) acts as a natural higher-rank generalization of the interval topology on linearly ordered sets and of the \textit{shadow topology} on pretrees~\cite{Mal14}. 
In the case of real pretrees, it is also 
closely related to Bowditch's order topology; see
\cite[Section~7]{B}.

For rank-one median algebras (that is, median pretrees) \(\tau_{\mathrm m}\) gives the well known  \textit{shadow topology}. The MMC \(\widehat X\) of weakly complete median pretrees  coincides with the end-compactification in the sense of Malyutin (Theorem~\ref{t:MedPretreeMMC}).

Proper median compactifications exhibit several distinct forms of behaviour. 
For some natural topologies, a proper median compactification need
not exist even in rank one: the metric hedgehog
$(X,m,\tau_d)$ of Corollary~\ref{c:0} is such an example. 

Another interesting phenomenon is the uniqueness property under natural assumptions.  
By Proposition~\ref{p:UniqueMMCCompactIntervals}, the MMC is the unique proper median compactification of
$(X,\tau_{\mathrm m})$ whenever $X\in(T^m_2)$ and every
interval is compact in the intrinsic topology $\tau_{\mathrm m}$. 
In Theorem~\ref{thm:unique_compactification}, we show that for \textit{algebraically discrete} median algebras  (equivalently algebras with finite intervals) equipped with the intrinsic topology $\tau_{\mathrm m}$, the MMC is the \textbf{unique}  proper median compactification. In particular, it coincides with the Roller compactification.

Proposition~\ref{p:RFandMMC} shows that if a topological median algebra $(X,m,\sigma)$ has compact intervals and satisfies $(T^m_2)$, then the Roller--Fioravanti (RF) compactification  is canonically isomorphic to the MMC of 
$(X,\tau_{\mathrm m})$. Note that RF compactification of $X$ is defined using canonical median retractions on all (compact) intervals.  
The intrinsic uniformity, however, is defined for arbitrary median
algebras using only canonical retractions onto chain intervals and
does not require compactness of all intervals. When this uniformity
is Hausdorff, its completion is characterized by a minimal universal
property among proper median compactifications of the specifically
topologized algebra $(X,\tau_{\mathrm m})$. Throughout the paper, the term
``minimal'' is understood in this relative sense. 
   
Combining Proposition~\ref{p:RFandMMC} with Fioravanti's
embedding theorem~\cite[Proposition~4.20]{Fioravanti20}, we obtain
conditions under which $\tau_{\mathrm m}=\tau_d$
(Theorem~\ref{t:coincidence}).
These conditions hold for every finite-dimensional locally
compact CAT(0) cube complex endowed with its cubical
$\ell^1$-metric $d$.   
   
   On the other hand, uniqueness fails in general, even in rank one. The ordered median algebra
   $X=\mathbb Q\cap[0,1]$, equipped with its intrinsic topology (equivalently, its order topology), admits many proper median compactifications, reflecting the classical nonuniqueness of compact LOTS extensions.
    
\sk 
Furthermore, we apply this machinery to actions of topological groups
on median algebras equipped with their intrinsic topology. In
Theorem~\ref{t:Gcompactification}, we show that if a topological group \(G\) acts 
(separately) continuously on \((X,\tau_{\mathrm m})\) by median automorphisms, then
\(\mathcal U_{\mathrm m}\) is \(G\)-bounded. Consequently, if
$\mathcal U_{\mathrm m}$ is Hausdorff, the action extends uniquely to a jointly continuous
action on the MMC, and the canonical map
\(
X\to \widehat X
\) 
is a proper median \(G\)-compactification. 
For finite-rank algebras \(X\), recent results from
\cite{Me-TameMedian} imply that the compact \(G\)-system
\(\widehat X\) is Rosenthal representable and hence dynamically tame. 
  
Theorem~\ref{t:Gcompactification} itself is proved independently in
the present paper. The companion preprint~\cite{Me-TameMedian},
which treats Roller--Fioravanti compactifications under
compact-interval assumptions, is used here only for the final
Rosenthal-representability and tameness conclusion in the
finite-rank case; see Remark~\ref{r:CompanionComparison}.  
For further discussion of the complexity hierarchy in dynamical systems and the important role of tame systems, we refer the reader to~\cite{GM-tLN,GM-D,GM-TC,Me-b}. 

The Dedekind, end- and Roller compactifications provide concrete
models of the MMC in several important cases. This suggests
the following problem.

\begin{problem} 
	\label{q:ExplicitMMC}
	Find explicit descriptions of the MMC $\widehat X$ for further
	natural classes of median algebras $X\in(T^m_2)$, in terms of
	intrinsic algebraic or geometric data.
	Of particular interest are classes whose intervals need not
	be compact in $\tau_{\mathrm m}$.
\end{problem}

\noindent\textbf{Acknowledgment.}
The author is very grateful to the anonymous referee for a
careful and constructive report, which led to several important
mathematical corrections and substantial improvements throughout
the paper. The referee also supplied the argument for
Proposition~\ref{p:UniqueMMCCompactIntervals}, which answers
a question posed in the previous version.

\section{Median algebras}  
\subsection*{Preliminaries}   

A \textit{median algebra} is a set $X$ with a ternary operation 
$m \colon X^3 \to X$ satisfying the standard median axioms.  
Frequently we write $abc$ instead of $m(a,b,c)$. 
Recall one possible system of axioms 
(see \cite{Sholander54,Vel-book,BowditchMedian}) defining median algebras:  

\begin{itemize}
	\item [(M1)] The operation $m$ is symmetric in all three variables.
	\item [(M2)] $abb=b$. 
	\item [(M3)] $(axb)xc=ax(bxc)$.      
\end{itemize}

A map $f \colon X_1 \to X_2$ between median algebras is said to be a \textit{homomorphism} or \textbf{median preserving} (MP) if 
$f(xyz)=f(x)f(y)f(z)$. Equivalently: 
for every convex subset \(C\subseteq X_2\), the preimage
\(f^{-1}(C)\) is convex in \(X_1\) (see, for example, \cite[Ch.1, 6.11]{Vel-book}). 

For every pair $x,y \in X$ we have the \textit{interval} $[x,y]_m:=\{z \in X: xyz=z\}$. 
Usually, we omit the subscript and write simply $[x,y]$, when the context is clear. Always, $[x,x]=\{x\}$, $[x,y]=[y,x]$. 

We say that a pair of distinct elements $a,b$ in $X$ is \textit{adjacent} if $[a,b]=\{a,b\}$. 

For every triple $x,y,z$  in $(X,m)$ and the intervals $[x,y], [y,z],  [x,z]$ we have a unique common point which is the median 
$
[x,y] \cap [y,z] \cap [x,z] = \{xyz\}. 
$   

For every pair $a,b\in X$ there is a natural partial order $\le_a$ on $[a,b]$ defined, for
$x,y\in [a,b]$, by 
\[
x\le_a y \quad \Longleftrightarrow \quad x\in [a,y].
\]
Equivalently, $x\le_a y$ iff $m(a,x,y)=x$. The opposite partial order is $\le_b$.

If $[a,b]$ is a chain, that is, if it admits a linear order $\preceq$ whose order-median coincides
with the median inherited from $X$, then $\preceq$ is necessarily either $\le_a$ or $\le_b$.
Indeed, since $[a,b]$ is the whole interval between $a$ and $b$, the points $a$ and $b$ must be
the two endpoints of the order $\preceq$. If $a$ is the least element, then
\[
x\preceq y \quad \Longleftrightarrow \quad m(a,x,y)=x
\quad \Longleftrightarrow \quad x\le_a y.
\]
If $b$ is the least element, then $\preceq=\le_b$.

A subset $C \subseteq X$ is \emph{convex} if $[x,y] \subseteq C$ for all $x,y \in C$. 
Every convex subset is a subalgebra. 
The intersection of convex subsets is convex.  

Several remarkable structures are median algebras under their natural medians. For instance,  distributive lattices
(e.g.\ linear orders, Boolean algebras, and power sets   $\mathcal P(S)$). 

If \(\leq\) is a linear order on \(X\), one defines the associated betweenness median \(m_\leq\) by declaring \(m_\leq(a,b,c)=b\) precisely when 
\(a\leq b\leq c\) or \(c\leq b\leq a\). 
A median algebra $(X,m)$ is said to be \textit{linear}
if there exists a linear order $\leq$ on $X$ such that $m=m_{\leq}$. 
In this case we also use the term \textit{chain}.

\begin{remark} \label{r:linear}
	For every linear median $(X,m)$	we have the following property: $m(x,y,z) \in \{x,y,z\}$ for every $x,y,z \in X$. This means that every nonempty subset of $X$ is a subalgebra. Some authors use the term \textit{conservative algebra}. Note that conservativity almost characterizes linear medians. More precisely, conservative median algebra is either linear or a $2$-hypercube $\{0,1\}^2$. For details see \cite{CMT} and \cite[p. 22]{BowditchMedian}.   
\end{remark}

A wall is a pair $W=\{W^0,W^1\}$ of nonempty disjoint convex subsets whose union is $X$.
The sets $W^0$ and $W^1$ are called halfspaces. 
Denote by $\mathcal{W}(X)$ and $\mathcal{H}(X)$ all walls and halfspaces of $X$. 

\begin{defin} \label{d:MedianRank} (see e.g., \cite{Vel-book,BowditchMedian,Fioravanti20})
	The \emph{rank} of a nontrivial median algebra $X$ (notation: $\rank(X)$) is the supremum of the numbers $n \in \N$ such that the Boolean hypercube $\{0,1\}^n$ embeds as a median subalgebra into $X$.   
	We adopt the convention that the one-point median algebra has rank \(0\).	
\end{defin}

The class of finite rank algebras is closed under taking subalgebras and finite products. The rank of the product $X_1 \times X_2$ of two median algebras is $\rank(X_1)+\rank(X_2)$. 
Important examples of algebras with rank $k \in \N$ are Boolean hypercubes $\{0,1\}^{k}$, usual cubes $[0,1]^{k}$ and 
CAT(0) cube complexes with dimension $k$. 

Onto homomorphisms cannot increase the rank \cite[Section 8.2]{BowditchMedian}. 
This property, combined with Remark \ref{r:linear}, implies the following observation.

\begin{lem} \label{l:HomImageChain}
	Let $X$ and $Y$ be median algebras, and let $f \colon X \to Y$ be a median homomorphism. If $C \subseteq X$ is a chain, then its homomorphic image $f(C)$ is a chain in $Y$.
\end{lem}
\begin{proof}
	Since $C$ is a chain, for every $a,b,c \in C$ we have
	$m(a,b,c) \in \{a,b,c\}$.
	Let $x,y,z \in f(C)$. Choose $a,b,c \in C$ such that
	$f(a)=x$, $f(b)=y$, and $f(c)=z$.
	Then
	\[
	m(x,y,z)=m(f(a),f(b),f(c))=f(m(a,b,c)) \in \{x,y,z\}.
	\]
	So $f(C)$ is a conservative median algebra. 
	Now the restriction $f|_C:C\to f(C)$ is an onto median homomorphism. Since
	onto homomorphisms do not increase rank and $\rank(C)\leq 1$, we get
	$
	\rank(f(C))\leq 1.
	$
	By Remark~\ref{r:linear}, a conservative median algebra is either linear or
	contains a $2$-cube. The latter is impossible because a $2$-cube has rank $2$.
	Hence $f(C)$ is linear, i.e. a chain.
\end{proof}

Rank-1 algebras are 
\textit{median pretrees} (in terms of Bowditch). In this case $\le_a$ is a linear ordering on any median interval $[a,b]$. It is a useful treelike structure which naturally generalizes linear orders and the betweenness relation on dendrons (e.g., dendrites), simplicial and $\R$-trees.  
Median preserving continuous group actions on median pretrees are investigated in \cite{GM-D}. 

A \emph{topological median algebra} (tma) 
is a topological space $(X,\tau)$ equipped with a continuous median operation $m \colon X^3 \to X$. 
Mainly we assume that every tma is Hausdorff. 
The only exception we allow is for intrinsic median topology $\tau_{\mathrm m}$ which need not be Hausdorff in general (though makes the median continuous). 

Many important examples of tma come from \textit{median metric spaces}, which play a major role in metric geometry and group theory; see, for example,  \cite{BowditchMedian, Fioravanti20,Vel-book}. 

For every median algebra $(X,m)$ and a pair of elements $u,v \in X$ we have a canonical median retraction 
$\phi_{u,v} \colon X \to [u,v]$, where $\phi_{u,v}(z) = m(u,z,v)$. If $X$ is a tma then every $\phi_{u,v}$ is continuous.   

A topological median algebra \(X\) is said to be \textit{locally convex} if every point of \(X\) has a neighborhood base consisting of convex sets.    
Median subalgebras of locally convex topological median algebras, endowed with the subspace topology, are locally convex. Arbitrary products of locally convex topological median algebras are locally convex as well.   
Every projection on each coordinate is a continuous MP map. 
In particular, the cube  $[0,1]^{\kappa}$ is a compact and locally convex median algebra for every  cardinal ${\kappa}$.  
   
	A compact locally convex median space $K$ is isomorphic as a topological median algebra to a subalgebra of the Tychonoff cube $[0,1]^{\kappa}$.  
	This can be derived from \cite[4.13.3 and  4.16]{Vel-book}.   
		If $K$ is zero-dimensional, then Tychonoff  cube can be replaced by the Cantor cube $\{0,1\}^{\kappa}$.  
	
Every compact Hausdorff tma of finite rank is locally convex (\cite[12.2.4 and 12.2.5]{BowditchMedian}).   

\begin{f} \label{f:facts2} Let $X$ be a Hausdorff topological median algebra and let $Y$ be its dense subalgebra. Then 
	\begin{enumerate} 		 		 
		\item \cite[Lemma 12.3.4]{BowditchMedian}  $\rank(Y)=\rank(X)$.
		\item \cite[Lemma 12.4.3]{BowditchMedian} If $Y$ is linear, then also $X$ is linear.  
	\end{enumerate}	
\end{f}

\begin{lem} \label{l:IntervalClosure}
	Let $K$ be a Hausdorff topological median algebra, and let $X$ be a dense median
	subalgebra of $K$. For any $u, v \in X$, the median interval $[u,v]_K$ in $K$
	is the topological closure of the interval $[u,v]_X$ in $K$. That is,
	\(
	[u,v]_K = \overline{[u,v]_X}^K.
	\) 
	
	Furthermore, if $[u,v]_X$ is a chain, then also $[u,v]_K$ is a chain.
\end{lem}
\begin{proof}
	Let $A=[u,v]_X$ and $B=[u,v]_K$. Since
	\(
	B=\{z\in K:\phi_{u,v}(z)=z\}
	\)
	and $K$ is Hausdorff, the set $B$ is closed in $K$. Since
	$A\subseteq B$, we obtain $\overline A^K\subseteq B$.
	
	Conversely, since $X$ is dense in $K$ and $\phi_{u,v}$ is
	continuous,
	\[
	B=\phi_{u,v}(K)
	=\phi_{u,v}(\overline X^K)
	\subseteq\overline{\phi_{u,v}(X)}^K
	=\overline A^K.
	\]
	Hence $[u,v]_K=\overline{[u,v]_X}^{\,K}$. 
	If $[u,v]_X$ is a chain, then $[u,v]_K$ is also a chain by
	Fact~\ref{f:facts2}(2), applied to the dense subalgebra
	$[u,v]_X$ of $[u,v]_K$.
\end{proof}

\subsection*{Median Uniformities and Compactifications}

In this section, we formalize the abstract relationship between uniform structures and compactifications in the category of median algebras.  
Recall some auxiliary definitions. 

For two covers $\alpha,\beta$ of a set $X$, we say that $\alpha$ \emph{refines} $\beta$ if for every $A \in \alpha$ there exists $B \in \beta$ such that $A \subseteq B$. 
Notation: $\alpha \succ \beta$.   
The \textit{star} of a subset $C \subset X$ with respect to $\a$ is the set 
$$
\St(C,\a)=\cup\{A \in \a: A \cap C \neq \emptyset\}. 
$$
If $C:=\{c\}$ is a singleton, we simply write $\St(c,\a)$. 
So, $\St(c,\a)=\cup\{A \in \a: c \in A\}$. 
The collection $\a^*:=\{\St(A,\a): A \in \a\}$ is a covering and is called the \textit{star of $\a$}. Always, 
$\a \succ \a^*$. 

For the standard definitions about uniform structures in terms of \textit{coverings} we refer to \cite{Isb,Eng89}.

\begin{defin} \label{d:cov-unif} \cite{Isb,Eng89}    
	Let $\mathcal{U}$ be a family of coverings on a set $X$. 
	Then $\mathcal{U}$ is said to be a \textit{(covering) uniformity} on $X$ if: 
	\begin{itemize}
		\item [(C1)] $\a,\beta \in \mathcal{U}$ implies that 
		$\a \wedge \beta:=\{A \cap B: A \in \a, B \in \beta\} \in \mathcal{U}$; 
		\item [(C2)] $\a \in \mathcal{U}$ and $\a \succ \beta$ imply that $\beta \in \mathcal{U}$; 
		\item [(C3)] every element in $\mathcal{U}$ has a star-refinement in $\mathcal{U}$ (meaning that for every  $\beta \in \mathcal{U}$ there exists $\a \in  \mathcal{U}$ such that $\a^* \succ \beta$).   
	\end{itemize} 
\end{defin}

A subfamily $\mathcal{B} \subset {\mathcal U}$ such that each $\eps \in {\mathcal U}$ has a refinement $\delta \succ  \eps$ with
$\delta \in \mathcal{B}$, is said to be a (uniform) \emph{base} of ${\mathcal U}$. A subfamily $\Gamma \subset {\mathcal U}$ is a \textit{subbase} (or, a \textit{prebase}) if the enriched family $\Gamma^{\wedge}:=\{\g_1 \wedge \cdots \wedge \g_n : \ \g_i \in \Gamma, n \in \N\}$ of common refinements is a base of ${\mathcal U}$.  
An abstract set $\mathcal{B}$ of coverings on $X$ is a base of some uniformity $\cU$ if and only if 
$$
	\forall \ P_1, P_2 \in \mathcal{B} \ \ \exists P_3 \in \mathcal{B} \ \ P_3^{*} \succ P_1 \wedge P_2. 
$$ 
Each uniformity $\mathcal{U}$ on $X$ defines a topology $top(\mathcal{U})$ on $X$: a subset $A \subseteq X$ is open iff for each $a \in A$ there exists a covering $P \in \mathcal{U}$ such that $\St(a,P) \subseteq A$. 

In terms of entourages (where $\Ucal$ is defined by entourages):  
for each $a \in A$ there exists an entourage (binary relation on $X$) $\eps \in \mathcal{U}$ such that $\eps(a):=\{x \in X: (a,x) \in \eps\} \subseteq A$. 

Below we consider mainly \textit{Hausdorff uniform structures}. This is the case when the uniform covers separate the points. That is, for every $x \neq y$ in $X$ there exists $\alpha \in \mathcal{U}$ such that $y \notin \St(x,\alpha)$. 

Recall that a covering uniform space $(X,\mathcal{U})$ is called \emph{precompact} (or, \textit{totally bounded}) if
$\mathcal{U}$ has a base consisting of finite covers of $X$. If the uniformity is Hausdorff, then it is equivalent to say that the uniform completion is compact. 

For median algebras $(X,m)$ we say that a uniformity $\mathcal U$ is \textit{convex} if there exists a uniform base $\mathcal B$ of $\mathcal U$ where every member $A \in \alpha$ of every $\alpha \in \mathcal B$ is convex (in usual median sense). 

\begin{defin}
	Let \(X\) be a topological median algebra. A \textit{median compactification} of \(X\)
	is a pair \((K,\eta)\), where \(K\) is a compact Hausdorff topological median
	algebra and
	$
	\eta \colon X\to K
	$
	is a continuous median-preserving map with dense image.
	
	We say that the compactification is \textit{proper} if \(\eta\) is a topological
	embedding.
\end{defin}

Consider the natural partial order on isomorphism classes of proper median compactifications of a topological median algebra $(X,m,\t)$, namely, by factor maps: 
\[
(Y_1,\eta_1)\preceq (Y_2,\eta_2)
\quad\Longleftrightarrow\quad
\exists\, q \colon Y_2\to Y_1
\text{ continuous surjective with } q \circ \eta_2=\eta_1.
\]
Then $q$ is automatically median-preserving. 

We order compatible precompact median uniformities by fineness. 
To systematically generate median compactifications, the underlying uniform structure must interact properly with the median operation.

\begin{defin} \label{d:median-compatible} 
	A Hausdorff uniformity $\mathcal{U}$ on a median algebra $X$ is called a \textbf{median-compatible uniformity} if the ternary median operation $m \colon X^3 \to X$ is uniformly continuous (where $X^3$ carries the product uniformity $\mathcal{U} \times \mathcal{U} \times \mathcal{U}$).  
\end{defin}

A fundamental source of precompact median-compatible  uniformities comes directly from proper compactifications: 
if $Y$ is a compact topological median algebra, its unique topologically compatible uniformity $\mathcal{U}_Y$ automatically makes $m_Y$ uniformly continuous. Thus, the restriction of $\mathcal{U}_Y$ to any dense median subalgebra $X$ naturally yields a precompact median-compatible uniformity on $X$. 

\begin{prop} \label{p:Median-compactif} 
	Let $X=(X,m,\tau)$ be a Hausdorff topological median algebra. There is a canonical
	order-preserving bijection between:
	\begin{enumerate}
		\item isomorphism classes of proper median compactifications of $X$;
		\item Hausdorff precompact median-compatible uniformities $\mathcal U$ on $X$ such that
		$\operatorname{top}(\mathcal U)=\tau$.
	\end{enumerate} 
\end{prop}

\begin{proof}
	Let $\mathcal U$ be a Hausdorff precompact median-compatible uniformity on $X$ with
	$\operatorname{top}(\mathcal U)=\tau$. Its Hausdorff uniform completion $\widehat X_{\mathcal U}$
	is compact. Since the median map $m \colon X^3\to X$ is uniformly continuous, it extends uniquely to
	a continuous map
	$\widehat m \colon \widehat X_{\mathcal U}^{\,3}\to \widehat X_{\mathcal U}.$
	The median identities pass to the completion by continuity and density.   
	Hence $\widehat X_{\mathcal U}$ is a compact Hausdorff tma, and the canonical
	map $X\to \widehat X_{\mathcal U}$ is a dense median embedding. Since $\operatorname{top}(\mathcal U)=\tau$, this embedding is topological. Thus we obtain a proper
	median compactification of $X$.
	
	Conversely, let $\eta:X\to K$ be a proper median compactification. The compact Hausdorff space
	$K$ carries its unique compatible uniformity $\mathcal U_K$. Since the median operation on $K$
	is continuous, it is uniformly continuous. Therefore the restriction of $\mathcal U_K$ to the dense
	subspace $\eta(X)$ induces, via $\eta$, a Hausdorff precompact median-compatible uniformity
	$\mathcal U_\eta$ on $X$. Since $\eta$ is a topological embedding, $\operatorname{top}(\mathcal U_\eta)=\tau$.
	
	These two constructions are inverse to each other by the standard compactification--uniformity
	correspondence. They also respect the orders: a factor map
	$q:K_2\to K_1$ with $q \circ \eta_2=\eta_1$ is equivalent to the inclusion
	$\mathcal U_{\eta_1}\subseteq \mathcal U_{\eta_2}$.
	Finally, such a factor map is automatically median-preserving, because it is median-preserving on
	the dense subalgebra $\eta_2(X)$ and both median operations are continuous.
\end{proof}

\begin{ex} \label{ex:MedianMetric}
	Let $(X, d)$ be a \textbf{median metric space} (such as a CAT(0) cube complex equipped with its $\ell_1$ metric, or an $\R$-tree). As shown in \cite[Corollary 2.15]{CDH}, the canonical median operation $m$ satisfies the Lipschitz inequality: 
	$ d(m(x_1, y_1, z_1), m(x_2, y_2, z_2)) \leq d(x_1, x_2) + d(y_1, y_2) + d(z_1, z_2)$
	Hence, the metric uniformity $\mathcal{U}_d$ is a median-compatible uniformity.  
\end{ex}

\section{Intrinsic uniformity and topology} 
\label{s:StarMedUn} 

In this section we construct a precompact convex uniformity
\(\mathcal U_{\mathrm m}\), relying entirely on the intrinsic algebraic
structure of the median operation. It induces an intrinsic,
possibly non-Hausdorff, uniformizable topology. 

We begin with the linearly ordered case. 
Let $(X,\leq)$ be a linearly ordered set. 
It carries the usual \textit{interval topology} $\lambda_{\leq}$, generated by the prebase of all open rays $(a,\infty)$ and $(-\infty,b)$. 
Recall the following recent characterization of the  uniformity for minimal proper order compactifications of a Linearly Ordered Topological Space (LOTS):

\begin{f} \label{f:DedekindUniformityInternal} 
\cite[Theorem 9.6]{Me-CircTop} 
	Let $(X,\leq,\lambda_{\leq})$ be a $\LOTS$. Let $\mathfrak{B}$ be the family of all finite ``star-covers'' of $X$, where for any finite chain $F = \{a_1 < a_2 < \dots < a_n\}$ in $X$, the star-cover $\mathcal{C}_F$ is defined as: 
	$$ \mathcal{C}_F := \{ (a_{i-1}, a_{i+1}) \mid i=1, \dots, n \},$$ 
	where $a_0=-\infty$ and $a_{n+1}=+\infty$. The uniformity $\mu_{\leq}$ generated by the base $\mathfrak{B}$ is precisely the Dedekind compactification uniformity. We call it the  \textbf{interval uniformity}.  	
\end{f}
	
	\begin{lem} \label{l:LOTSSubbase}
		The interval uniformity $\mu_{\leq}$ on a $\LOTS$ $(Y,\leq)$ admits a uniform subbase consisting of all two-element covers of the form $\mathcal{O}_{u,v} = \{ L_v, R_u \}$, where $u < v$ and $L_v:=(-\infty,v), R_u:=(u,\infty)$.
	\end{lem}
	\begin{proof}
		Let $F = \{a_1 < a_2 < \dots < a_n\}$ be a finite chain in $Y$. The corresponding basic cover in $\mu_{\leq}$ is $\mathcal{C}_F = \{ (a_{i-1}, a_{i+1}) \mid i=1, \dots, n \}$, where $a_0 = -\infty$ and $a_{n+1} = +\infty$.
		For each adjacent pair $a_i, a_{i+1}$ in $F$, consider the two-element cover $\mathcal{O}_{a_i, a_{i+1}} = \{ L_{a_{i+1}}, R_{a_i} \}$. 
		We form the common refinement (wedge intersection) of these $n-1$ covers: $\mathcal{W} = \bigwedge_{i=1}^{n-1} \mathcal{O}_{a_i, a_{i+1}}$. 
		
		An element of $\mathcal{W}$ is an intersection $\bigcap_{i=1}^{n-1} P_i$, where each $P_i \in \{L_{a_{i+1}}, R_{a_i}\}$. Because $L_{a_{i+1}} \cap R_{a_{i+1}} = (-\infty, a_{i+1}) \cap (a_{i+1}, +\infty) = \emptyset$, any sequence of choices containing $L_{a_{i+1}}$ followed by $R_{a_{i+1}}$ results in an empty intersection. Thus, the only non-empty elements of $\mathcal{W}$ arise from choosing a sequence of right rays $R_{a_j}$ up to some index $k-1$, and left rays $L_{a_m}$ for all subsequent indices. The intersection of such a sequence is $R_{a_{k-1}} \cap L_{a_{k+1}} = (a_{k-1}, +\infty) \cap (-\infty, a_{k+1}) = (a_{k-1}, a_{k+1})$. 
		
		This demonstrates that the non-empty elements of $\mathcal{W}$ are the sets $(a_{i-1}, a_{i+1})$. Therefore, the finite wedge intersection of the two-element covers yields the basic cover $\mathcal{C}_F$. 
		Conversely, every two-ray cover \(\mathcal O_{u,v}\) belongs to \(\mu_{\leq}\), since it is the star-cover associated with the two-point chain \(\{u<v\}\).
		 Consequently, by Fact \ref{f:DedekindUniformityInternal}, the family of all covers $\mathcal{O}_{x,y}$ forms a uniform subbase for $\mu_{\leq}$.
	\end{proof}
	
	\subsection*{Connection to the shadow (branch) topology}
	
	Let $X$ be a \textit{pretree} in the sense of \cite{B} with its betweenness ternary relation. For every pair $a,b \in X$ denote by $[a,b]$ the corresponding interval, consisting of all points $x \in X$ which are between $a$ and $b$.  
	
	Following A.V. Malyutin \cite{Mal14} (which in turn follows the terminology of P. de la Harpe and J.-P. Preaux), we define the so-called \textit{shadow topology}. Alternative names in related structures are: \textit{Lawson's topology} and \textit{observer's topology}. See the related discussion in \cite{Mal14}.  
	Given an ordered pair $(u,v) \in X^2$ with $u \neq v$, let 
	$$
	S^v_u := \{x \in X \mid u \in [x,v]\}
	$$
	be the \textit{shadow} in $X$ defined by the ordered pair $(u,v)$. 
	Pictorially, the shadow $S^v_u$ is cast by a point $u$ when the light source is located at the point $v$. 
	The family $\mathcal{S} = \{S^v_u \mid u,v \in X, u \neq v\}$ is a subbase for the closed sets of the topology $\tau_s$. 
	The complement of $S^v_u$ is said to be a \textit{branch}:
	$$B^v_u := X \setminus S^v_u = \{x \in X \mid u \notin [x,v]\}.$$
	The set of all branches $\{B^v_u \mid u, v \in X, u \neq v\}$ is a subbase of the shadow topology on the pretree $X$. 
	
	\sk   
	In the case of a linearly ordered set, $B^v_u=(u,\infty)$ for $u<v$ and $B^v_u=(-\infty,u)$ for $v<u$. Hence, we get the \textit{interval topology}. 
	In general, for an abstract pretree, the shadow topology is often (but not always) Hausdorff. 
	Furthermore, by \cite[Theorem 7.3]{Mal14}, a pretree equipped with its shadow topology is Hausdorff if and only if, as a topological space, it can be embedded into a dendron. 
	
	We introduce the following \textit{median versions} of shadows and branches.  
	 
\begin{defin} \label{d:MedianBranch} 
Let $(X,m)$ be a median algebra. For every ordered pair $u,v \in X$ define: 
\begin{itemize}
	\item \textit{Shadow} $S^v_u := \{x \in X \mid u = m(u,x,v)\}= \{x \in X \mid u \in [x, v]\}=\phi_{u,v}^{-1}(\{u\})$. 
	\item \textit{Branch} $B^v_u = X \setminus S^v_u= \{ x \in X \mid m(u,x,v) \neq u \}=\{x \in X \mid u \notin [x, v]\}=\phi_{u,v}^{-1}([u,v]\setminus\{u\})$. 
\end{itemize} 	
\end{defin}
	 
Here $\phi_{u,v} \colon X \to [u,v]$ is the canonical median retraction $\phi_{u,v}(z) = m(u,z,v)$.  
So, for $u \neq v$ we obtain a two-member cover of $X$:
$$\mathcal B_{u,v}:=\{B^u_v, B^v_u\}=\mathcal B_{v,u}.$$    

Later we use the following equivalent descriptions. For 
\(p=\phi_{u,v}(x)=m(u,x,v)\), one has
\[
x\in B^u_v \quad \Longleftrightarrow \quad p\neq v,
\qquad
x\in B^v_u \quad \Longleftrightarrow \quad p\neq u.
\]
Equivalently,
\[
B^u_v=\phi_{u,v}^{-1}([u,v]\setminus\{v\}),\qquad
B^v_u=\phi_{u,v}^{-1}([u,v]\setminus\{u\}).
\]
In particular, if \([u,v]\) is a chain ordered by \(\leq_u\), then
\[
B^u_v=\phi_{u,v}^{-1}([u,v)),\qquad
B^v_u=\phi_{u,v}^{-1}((u,v]).
\]

	\begin{lem} \label{l:ConvexFibers}
		Let $X$ be a median algebra and let $[u,v] \subseteq X$ be a chain, where $u \neq v$.  Then the branch  
		$B^u_v$ is a halfspace (hence, convex) in $X$.   
	\end{lem}
	\begin{proof} 	
		By definition $B^u_v$ is the preimage $\phi_{u,v}^{-1}([u, v))$, and $[u, v) = \{w \in [u,v] \mid w \neq v\}$.  
		Since $[u,v]$ is a linearly ordered median algebra (a chain), any sub-interval that excludes an endpoint is convex. Therefore, $[u, v)$ is a convex subset of $[u,v]$. 
		Furthermore, the complement of $[u, v)$ in $[u,v]$ is the singleton $\{v\}$, which is also trivially convex.  
		Because the preimage of a convex set under a median homomorphism is always convex, both $B^u_v = \phi_{u,v}^{-1}([u, v))$ and its complement $X \setminus B^u_v = S^u_v=\phi_{u,v}^{-1}(\{v\})$ are convex sets in $X$. 
		Thus, $B^u_v$ is a halfspace.  
	\end{proof}
	
	\begin{remark} \label{r:chain_essential}
		The assumption in Lemma \ref{l:ConvexFibers}  that $[u,v]$ is a \textbf{chain}, is essential. For a general interval $[u,v]$ in a median algebra of rank $\geq 2$, the semi-open interval $[u,v) = [u,v] \setminus \{v\}$ might be non-convex (so its preimage $\phi_{u,v}^{-1}([u,v))$ would also fail to be a convex set in general.)
		
		To see this, consider the standard 2-hypercube $X = \{0,1\}^2$. Let $u = (0,0)$ and $v = (1,1)$. The interval $[u,v]$ is the entire space $X$, and $[u,v) = \{(0,0), (1,0), (0,1)\}$. If we choose the points $x = (1,0)$ and $y = (0,1)$, both clearly belong to $[u,v)$. However, $m(x, v, y) = m((1,0), (1,1), (0,1)) = (1,1) = v$. So $v \in [x,y]$, and the interval $[x,y]=X$ is not contained in $[u,v)$. Therefore $B^u_v=[u,v)$ is not convex in $X$.   
	\end{remark}
	
	There are two complementary ways to view the uniformity introduced below. 
	The first is purely algebraic and uses the two branch halfspaces determined by chain intervals.
	The second, proved in Theorem~\ref{t:InitialUniformity}, identifies the same uniformity as the
	initial uniformity generated by the canonical gate retractions onto chain intervals equipped with
	their interval uniformities using Fact \ref{f:DedekindUniformityInternal}. 
	
	\begin{defin} \label{d:StarMedianUniformity}
		Let \(X\) be a median algebra. Consider the family of two-member
		covers
		\[
		\mathcal B_{x,y}:=\{B^x_y,B^y_x\},
		\]
		where \([x,y]\) runs over all nontrivial chain intervals in \(X\).
		As shown in Theorem~\ref{t:InitialUniformity}, this family is a
		uniform subbase. The resulting covering uniformity is called the
		\textbf{median uniformity}, or \textbf{\(m\)-uniformity}, and is
		denoted by \(\mathcal U_m\). 
		The intrinsic \textbf{median topology}, or \textbf{\(m\)-topology},
		denoted by \(\tau_{\mathrm m}\), is the topology induced by
		\(\mathcal U_m\).
	\end{defin}
	
	The requirement that $[x,y]$ is a (nontrivial) chain is critical;  it ensures that the branches $B^x_y$ are convex. So, each member of the cover $\mathcal B_{x,y}$ is convex. Since intersection of convex subsets is convex, this guarantees that  $\mathcal{U}_{\mathrm{m}}$ contains a uniform base which consists of convex covers. 
	
	The topology $\tau_{\mathrm m}$ serves as the natural median analogue to the classical interval topology defined on abstract linearly ordered sets.

\begin{lem} \label{l:SubbasicStars} 
	Let $X$ be a median algebra and $[u,v]\subseteq X$ be a chain, where $u \neq v$. 
	For every $x\in X$, exactly one of the cases {\rm (1), (2), (3)} occurs. Moreover: 
	\begin{enumerate}
		\item $\phi_{u,v}(x)=u$. 
		In this case
		$x\in B^u_v \setminus B^v_u
		\ \text{and} \ 
		\St(x,\mathcal B_{u,v})=B^u_v.$
		
		\item $\phi_{u,v}(x)=v$. In this case
		$
		x\in B^v_u \setminus B^u_v
		\ \text{and} \
		\St(x,\mathcal B_{u,v})=B^v_u.
		$
		
		\item
		If \(\phi_{u,v}(x)\notin\{u,v\}\), then 
		$
		x \in B^u_v \cap B^v_u 
		\ \text{and}\
		\St(x,\mathcal B_{u,v})=X.
		$
		
		\item $\mathcal{B}_{u,v} = \{B^u_v, B^v_u\}$ is the disjoint cover if and only if $u,v$ is an adjacent pair in $X$. 
	\end{enumerate}
\end{lem} 
\begin{proof}
	(1) If $\phi_{u,v}(x)=u$, then by definition $m(u,x,v)\neq v$, hence $x\in B^u_v$, while
	$m(u,x,v)=u$, so $x\notin B^v_u$. Therefore $x$ belongs only to the member $B^u_v$ of the
	two-element cover $\mathcal B_{u,v}$, and thus
	$
	\St(x,\mathcal B_{u,v})=B^u_v.
	$ 
	
	(2) The case $\phi_{u,v}(x)=v$ is symmetric. 
	
	(3) If $\phi_{u,v}(x)\notin\{u,v\}$, then $m(u,x,v)\neq u$ and $m(u,x,v)\neq v$, so
	$
	x\in B^u_v \cap B^v_u.
	$
	Hence $x$ belongs to both members of the cover $\mathcal B_{u,v}$, and therefore
	$
	\St(x,\mathcal B_{u,v})=B^u_v \cup B^v_u=X.
	$
	
	(4) Since $\phi_{u,v}$ is onto $[u,v]$, the intersection
	$B^u_v\cap B^v_u$ is empty if and only if
	$[u,v]\setminus\{u,v\}$ is empty. This is equivalent to
	$[u,v]=\{u,v\}$, that is, to $u$ and $v$ being adjacent. 
\end{proof}

\begin{lem} \label{l:compositionOFgates} 
	Let $X$ be a median algebra and $x,y,u,v \in X$. Consider the canonical gate retractions 
	$$\phi_{x,y} \colon X \to [x,y], \ \phi_{u,v} \colon X \to [u,v].$$ 
	Denote $c:=\phi_{u,v}(x), d:=\phi_{u,v}(y)$. Then:
	\begin{enumerate}
		\item The composition $\phi_{u,v} \circ \phi_{x,y}$ coincides with the canonical retraction $\phi_{c,d} \colon X \to [c,d]$.  
		
		In particular, if \([x,y]\subseteq[u,v]\), then
		\(c=x\), \(d=y\), and
		\(
		\phi_{u,v}\circ\phi_{x,y}=\phi_{x,y}.
		\) 
		
\item If, in addition,
$\phi_{x,y}(u)=x$ and $\phi_{x,y}(v)=y$, then
\[
\phi_{u,v}|_{[x,y]}\colon [x,y]\to[c,d]
\quad\text{and}\quad
\phi_{x,y}|_{[c,d]}\colon[c,d]\to[x,y]
\]
are mutually inverse median isomorphisms. In particular,
$\phi_{x,y}(c)=x$ and $\phi_{x,y}(d)=y$.  
\end{enumerate}  		
\end{lem}
\begin{proof}  
	(1) Put $A=[u,v]$, $B=[x,y]$, and denote the corresponding
gate retractions by $\omega_A=\phi_{u,v}$ and
$\omega_B=\phi_{x,y}$. 
Recall that, by \cite[Lemma~7.1.5]{BowditchMedian}, every
surjective median homomorphism $f\colon M\to N$ maps intervals
onto intervals; that is,
\(
f([a,b]_M)=[f(a),f(b)]_N
\)
for all $a,b\in M$. Applying this to the surjective median
homomorphism $\omega_A\colon X\to A$, we obtain
\[
A_B:=\omega_A(B)
=\phi_{u,v}([x,y])
=[\phi_{u,v}(x),\phi_{u,v}(y)]_A
=[c,d]_A.
\]
Since $A=[u,v]$ is convex in $X$, one has
$[c,d]_A=[c,d]_X$. Hence $A_B=[c,d]$.

By \cite[Lemma~7.3.3]{BowditchMedian}, the composition of the
gate retractions onto $A$ and $B$ is the gate retraction onto
the projection $A_B$. Therefore
\(
\omega_A\circ\omega_B=\phi_{c,d},
\)
or equivalently,
\(
\phi_{u,v}\circ\phi_{x,y}=\phi_{c,d}.
\) 
If $[x,y]\subseteq[u,v]$, then $\phi_{u,v}$ fixes
$[x,y]$ pointwise. Hence $c=x$ and $d=y$, and therefore
\(
\phi_{u,v}\circ\phi_{x,y}=\phi_{x,y}.
\)

(2) Retain the notation $A_B=[c,d]$ from part~(1), and put
$B_A:=\omega_B(A)$.
Applying part~(1) in both directions, we have
	\(
	\omega_A\circ\omega_B=\phi_{A_B}
	\)
	and
	\(
	\omega_B\circ\omega_A=\phi_{B_A}.
	\)
	Since these retractions are the identity on their images, the
	restrictions
	\(
	\omega_A|_{B_A} \colon B_A\to A_B
	\)
	and
	\(
	\omega_B|_{A_B} \colon A_B\to B_A
	\)
	are mutually inverse median isomorphisms.

	By the assumptions,
	$
	B_A=[\phi_{x,y}(u),\phi_{x,y}(v)]=[x,y],
	$
	while \(A_B=[c,d]\). Therefore, the restrictions
	\[
	\omega_A|_{[x,y]}\colon [x,y]\to[c,d]
	\quad\text{and}\quad
	\omega_B|_{[c,d]}\colon[c,d]\to[x,y]
	\]
	are mutually inverse median isomorphisms. Since
	\(\omega_A(x)=c\) and \(\omega_A(y)=d\), the inverse relation gives
	$
	\omega_B(c)=x,\ \omega_B(d)=y.
	$
\end{proof}

\begin{lem} \label{lem:branch_equiv} 
Let \(X\) be a median algebra and let \([x,y]\) be a \textbf{chain} interval, equipped with the order \(\leq_x\). Then: 
	\begin{enumerate}
		\item If \(z\in [x,y)\), then
		\(
		B^x_y=B^z_y.
		\) 
		
		\item If \(z\in (x,y]\), then
		$
		B^y_x=B^z_x.
		$
	\end{enumerate}
	
	In particular, if \(z\in (x,y)\), then both equalities hold.
\end{lem} 
\begin{proof}
	We prove (1); the proof of (2) is symmetric. Assume that \(z\in[x,y)\).
	It is enough to show that \(S^x_y=S^z_y\), or equivalently, by
	Definition~\ref{d:MedianBranch}, that for every \(w\in X\),
	\[
	y\in[x,w]\quad \Longleftrightarrow\quad y\in[z,w].
	\]
	
	Let \(p=\phi_{x,y}(w)=m(x,w,y)\). Then \(p\in[x,y]\). By the
	definition of a median interval,
	\[
	y\in[x,w]
	\quad\Longleftrightarrow\quad
	m(x,w,y)=y.
	\]
	Since \(p=m(x,w,y)\), this gives
	$
	y\in[x,w]\quad\Longleftrightarrow\quad p=y.
	$
	
	Since \([z,y]\subseteq[x,y]\), the canonical retractions satisfy
	$
	\phi_{z,y}=\phi_{z,y}\circ\phi_{x,y}.
	$
	Indeed, this follows from Lemma~\ref{l:compositionOFgates}(1), applied to the
	nested intervals \([z,y]\subseteq[x,y]\). Hence
	\[
	\phi_{z,y}(w)=\phi_{z,y}(\phi_{x,y}(w))=\phi_{z,y}(p).
	\]
	
	Since \([x,y]\) is a chain and \(z<_x y\), the retraction \(\phi_{z,y}\)
	sends a point \(p\in[x,y]\) to \(y\) only when \(p=y\). Indeed, if
	\(p\leq_x z\), then \(\phi_{z,y}(p)=z\), while if \(z<_x p<_x y\), then
	\(\phi_{z,y}(p)=p\). Therefore
	\[
	\phi_{z,y}(p)=y \quad \Longleftrightarrow \quad p=y.
	\]
	
	Hence,
	\(
	y\in[z,w]
	\quad \Longleftrightarrow \ \
	\phi_{z,y}(w)=y
	\quad \Longleftrightarrow \ \ 
	\phi_{z,y}(p)=y
	\quad \Longleftrightarrow \ \ 
	p=y
	\quad \Longleftrightarrow \ \ 
	y\in[x,w].
	\) 
\end{proof}

\begin{ex}[Importance of \textbf{chain} intervals] \label{ex:ChainEssentiality} 
	We show that the equality \(B^x_y=B^z_y\) from
	Lemma~\ref{lem:branch_equiv} may fail if the interval \([x,y]\) is not a
	chain.  
	Consider the \(3\)-cube \(X=\{0,1\}^3\) with the coordinatewise median. Let
	\[
	x=(0,0,0),\qquad y=(1,1,0),\qquad z=(1,0,0).
	\]
	Then \(z\in[x,y]\), but \([x,y]\) is the square
	$
	[x,y]=\{0,1\}\times\{0,1\}\times\{0\},
	$
	and therefore is not a chain. 	 
	Let us compute the relevant shadows and branches. By Definition~\ref{d:MedianBranch}, 
	\[
	S^x_y=\{t\in X: y \in [t,x]\},\qquad B^x_y=X \setminus S^x_y.
	\]
	For the pair \(x,y\), one has
	$
	S^x_y=\{(1,1,0),(1,1,1)\},
	$
	and hence
	\[
	B^x_y
	=
	X\setminus S^x_y
	=
	\{(0,0,0),(1,0,0),(0,1,0),(0,0,1),(1,0,1),(0,1,1)\}.
	\]
	For the pair \(z,y\), we have
	$
	S^z_y
	=
	\{(0,1,0),(1,1,0),(0,1,1),(1,1,1)\}, 
	$
	and therefore
	\[
	B^z_y
	=
	X\setminus S^z_y
	=
	\{(0,0,0),(1,0,0),(0,0,1),(1,0,1)\}  \neq B^x_y.
	\]  
			\begin{figure}[h]  
		\begin{center}   
			\scalebox{0.18}{\includegraphics{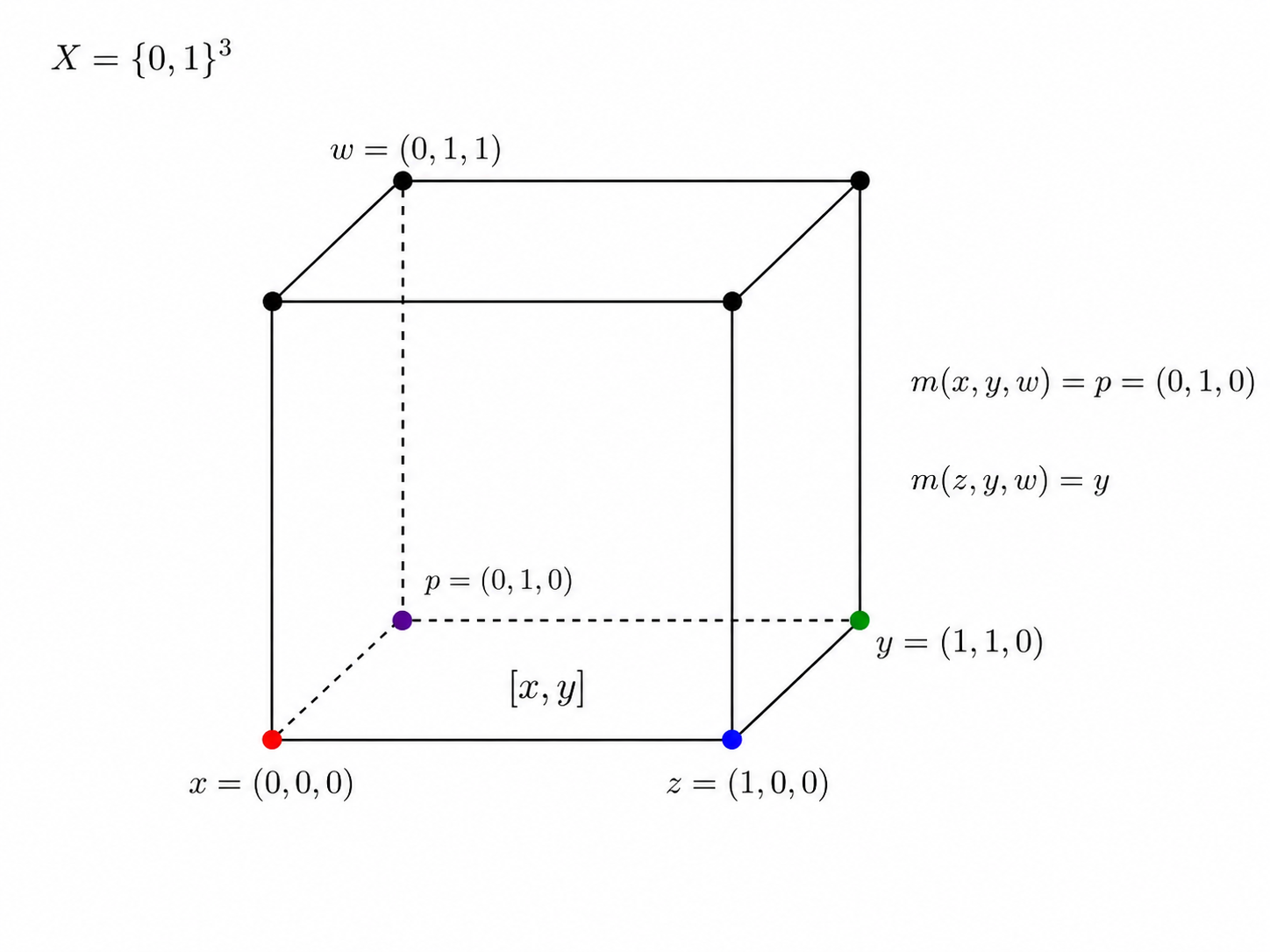}}
			\caption{Importance of chain intervals}
		\end{center} 
		\label{fig:first}
	\end{figure} 
	\end{ex}

\begin{lem}\label{l:PreimageRays}
	Let $X$ be a median algebra and let $[u,v]$ be a chain
	interval in $X$, ordered by $\leq_u$. Let
	\(
	\phi_{u,v}\colon X\to [u,v],\ 
	\phi_{u,v}(x)=m(u,x,v),
	\) 
	be the canonical retraction. Let
	\(
	u\leq_u c<_u d\leq_u v,
	\) 
	and put
	\(
	L_d:=\{t\in[u,v]:t<_u d\},
	\
	R_c:=\{t\in[u,v]:c<_u t\}.
	\)
	Then
	\begin{enumerate}
		\item $\phi_{u,v}^{-1}(L_d)=B^u_d$;
		\item $\phi_{u,v}^{-1}(R_c)=B^v_c$.
	\end{enumerate}
	Consequently,
	\(
	\phi_{u,v}^{-1}\bigl(\{L_d,R_c\}\bigr)
	=
	\{B^u_d,B^v_c\}.
	\)  
\end{lem} 
\begin{proof}
	We prove (1). Let \(a\in X\) and put
	$
	p=\phi_{u,v}(a)=m(u,a,v)\in [u,v].
	$
	Since \([u,d]\subseteq [u,v]\), the nested-retraction identity applying Lemma \ref{l:compositionOFgates}(1) gives
	$
	\phi_{u,d}(a)=\phi_{u,d}(\phi_{u,v}(a))=\phi_{u,d}(p).
	$ 
	
	On the chain \([u,v]\), the retraction \(\phi_{u,d}\) sends \(p\) to \(p\) if
	\(p\leq_u d\), and to \(d\) if \(d\leq_u p\). Hence
	\[
	p<_u d
	\quad\Longleftrightarrow\quad
	\phi_{u,d}(p)\neq d
	\quad\Longleftrightarrow\quad
	\phi_{u,d}(a)\neq d.
	\]
	By the definition of branches, the last condition is precisely \(a\in B^u_d\).
	Thus \(\phi_{u,v}^{-1}(L_d)=B^u_d\).
	
	(2) The proof for the right ray $R_c$ is symmetric.   
	Since $c<_u d$, the family $\{L_d,R_c\}$ covers $[u,v]$.
	Therefore, by (1) and (2),
	$
	\phi_{u,v}^{-1}\bigl(\{L_d,R_c\}\bigr)
	=
	\{B^u_d,B^v_c\}.
	$
\end{proof}

\begin{thm} \label{t:InitialUniformity}
	The intrinsic median uniformity $\mathcal{U}_m$ on $X$ coincides with the initial (weak)  uniformity defined by the family of all canonical gate retractions $\phi_{u,v} \colon X \to ([u,v], \mu_{\le})$ onto chain intervals. Consequently, the intrinsic median topology $\tau_{\mathrm m}$ coincides with the initial topology defined by these retractions onto the chain intervals equipped with their interval topologies $([u,v], \lambda_{\le})$.
\end{thm}

\begin{proof}  
	Let \(\mathcal U_{\mathrm{initial}}\) denote the initial uniformity
	generated by the canonical retractions
	\[
	\phi_{u,v} \colon X\longrightarrow([u,v],\mu_{\leq}),
	\]  
	where \([u,v]\) runs over all nontrivial chain intervals of \(X\) and $\mu_{\leq}$ is the interval uniformity (Fact \ref{f:DedekindUniformityInternal}).
	We show that the family of branch covers
	\(
	\mathcal B_{x,y}=\{B^x_y,B^y_x\}
	\) 
	forms a uniform subbase for \(\mathcal U_{\mathrm{initial}}\).
	
	To see that $\mathcal{U}_m \subseteq \mathcal{U}_{initial}$, recall that $\mathcal{U}_m$ is generated by the subbasic covers $\mathcal{B}_{x,y} = \{B_y^x, B_x^y\}$ for all chain intervals $[x,y]$. By Lemma \ref{l:PreimageRays}, we have $B_y^x = \phi_{x,y}^{-1}(L_y)$ and $B_x^y = \phi_{x,y}^{-1}(R_x)$. By Lemma \ref{l:LOTSSubbase}, $\mathcal{O}_{x,y} = \{L_y, R_x\}$ is a subbasic uniform cover in $([x,y], \mu_\le)$. Therefore, its preimage $\mathcal{B}_{x,y}$ belongs to $\mathcal{U}_{initial}$ by definition, which yields $\mathcal{U}_m \subseteq \mathcal{U}_{initial}$.
	
	Conversely, $\mathcal{U}_{initial}$ is generated by the pullbacks of subbasic uniform covers of the target chain intervals. By Lemma \ref{l:LOTSSubbase}, a uniform subbase for $([u,v], \mu_\le)$ is given by covers of the form $\mathcal{O}_{c,d} = \{L_d, R_c\}$, where $u \le_u c <_u d \le_u v$. By Lemma \ref{l:PreimageRays}, their preimages under $\phi_{u,v}$ are exactly $\phi_{u,v}^{-1}(L_d) = B_d^u$ and $\phi_{u,v}^{-1}(R_c) = B_c^v$.

	Since \(c\in [u,d)\), Lemma~\ref{lem:branch_equiv} applied to the chain
	\([u,d]\) gives \(B^u_d=B^c_d\). Similarly, since \(d\in(c,v]\), the same lemma
	applied to the chain \([c,v]\) gives \(B^v_c=B^d_c\).
	Consequently, 
	\[
	\phi_{u,v}^{-1}(\mathcal O_{c,d})
	=
	\{B^u_d,B^v_c\}
	=
	\{B^c_d,B^d_c\}
	=
	\mathcal B_{c,d}.
	\] 
	 Since $[c,d]$ is a chain interval in $X$, this cover belongs to $\mathcal{U}_m$, proving that $\mathcal{U}_{initial} \subseteq \mathcal{U}_m$.
	
	Finally, it is a standard result in general topology that the uniform topology of an initial uniformity is exactly the initial topology defined by the same family of maps. Since the interval uniformity $\mu_\le$ induces the interval topology $\lambda_\le$, it follows immediately that $\tau_{\mathrm m}$ is the initial topology generated by the retractions $\phi_{u,v} \colon X \to ([u,v], \lambda_{\le})$. Alternatively, this follows also directly by Lemma \ref{l:PreimageRays}. 
\end{proof}

\begin{cor} \label{c:BranchesSubbaseTauM}
	Let \(X\) be a median algebra. Then 
	\begin{enumerate}
		\item The family of all intrinsic branches
		$B^u_v$, $B^v_u,$
		where \([u,v]\) runs over all nontrivial chain intervals, is an open subbase for
		the topology \(\tau_{\mathrm m}\). 
		\item \(\tau_{\mathrm m}\) is locally convex.
	\end{enumerate} 
\end{cor}

\begin{proof} (1) Directly follows from Lemma \ref{l:PreimageRays} and Theorem  \ref{t:InitialUniformity}. 
	
	(2) 
	Every branch \(B^u_v\) is convex by Lemma~\ref{l:ConvexFibers}, and finite
	intersections of convex sets are convex. Therefore \(\tau_{\mathrm m}\) has a base
	consisting of convex open sets. 
\end{proof}

\begin{defin} \label{d:medT2}
	A median algebra $X$ is said to satisfy the \emph{median $T_2$ property}  if for every two distinct
	points $x,y\in X$ there exists a chain interval $[u,v]\subseteq X$ such that
	$
	\{\phi_{u,v}(x),\phi_{u,v}(y)\}=\{u,v\}.
	$
	
	We use the notation: $X \in (T^m_2)$. 
\end{defin}

\begin{lem} \label{l:medT2Hausdorff}
	For a median algebra $X$, the following are equivalent:
	\begin{enumerate}
		\item $\mathcal U_{\mathrm{m}}$ is Hausdorff.
		\item For every distinct points $x,y\in X$ there exists a subbasic cover
		$\mathcal B_{u,v}$ such that
		\(
		y\notin \St(x,\mathcal B_{u,v}).
		\)
		\item $X \in (T^m_2)$. 
	\end{enumerate}
\end{lem}

\begin{proof}
	\emph{(2)$\Rightarrow$(1)} is immediate.
	
	\emph{(1)$\Rightarrow$(2)}:
	Let $x\neq y$ in $X$. Since $\mathcal U_{\mathrm{m}}$ is Hausdorff, there exists
	$\alpha\in\mathcal U_{\mathrm{m}}$ such that
	\(
	y\notin \St(x,\alpha).
	\) 
	Because $\mathcal U_{\mathrm{m}}$ is generated by the subbase of all covers
	$\mathcal B_{u,v}$, there exists a finite common refinement 
	\[
	\beta=\mathcal B_{u_1,v_1}\wedge \cdots \wedge \mathcal B_{u_n,v_n}
	\]
	such that $\beta\succ \alpha$.
	Hence
	\( 
	\St(x,\beta)\subseteq \St(x,\alpha),
	\)
	so in particular $y\notin \St(x,\beta)$.
	
	For a finite wedge refinement, the distributive law for sets implies:
	$
	\St(x,\beta)=\cap_{i=1}^n \St(x,\mathcal B_{u_i,v_i}).
	$
	Since  
	$
	y\notin \bigcap_{i=1}^n \St(x,\mathcal B_{u_i,v_i}),
	$
	there exists some index $i$ such that  
	$y\notin \St(x,\mathcal B_{u_i,v_i}).$
	This proves (2).
	
	\emph{(2)$\Leftrightarrow$(3)}:
	By Lemma~\ref{l:SubbasicStars}, the condition
	$
	y\notin \St(x,\mathcal B_{u,v})
	$
	holds if and only if one of the two projections is $u$ and the other is $v$, i.e.,
	$
	\{\phi_{u,v}(x),\phi_{u,v}(y)\}=\{u,v\}.
	$
	This is exactly Definition~\ref{d:medT2}.
\end{proof}

\begin{cor} \label{l:TychTop} 
	$\tau_{\mathrm{m}}$ is Tychonoff (completely regular and Hausdorff) if and only if $X \in (T^m_2)$.  
\end{cor}

We show in Section \ref{s:FiniteAndDiscrete}  that the m-uniformity is Hausdorff for important classes of median algebras: for finite-rank algebras (Theorem \ref{t:FiniteRankHausd}) and for algebraically discrete algebras (Theorem \ref{t:DISCRETEisCV}). 
In general, the intrinsic uniformity \(\mathcal U_m\) and the m-topology $\tau_{\mathrm m}$ need not be Hausdorff.
It may even be the trivial uniformity. This happens precisely when \(X\) has no nontrivial chain intervals.   

\begin{ex}[Failure of $(T_2^m)$: a ``Swiss cheese" algebra] \label{ex:PERIODIC} 
	A useful example is the algebra \(P\) of periodic binary bi-sequences;
	see \cite[p.~63]{BowditchMedian}. Namely, \(P\) is the subalgebra of the Cantor
	cube \(\{0,1\}^{\mathbb Z}\) consisting of all periodic bi-infinite binary
	sequences. If \(x\neq y\) in \(P\), then the interval \([x,y]_P\) is not a
	chain. Indeed, as observed by Bowditch, such intervals contain finite
	hypercubes of arbitrarily large dimension. 
	Hence \(P\) has no nontrivial chain intervals.  
	Therefore \(\mathcal U_m(P)\) is trivial.
\end{ex}

\begin{lem} \label{l:LeastToplogy} 
	Let $(X,m)$ be a median algebra and $\sigma$ be a Hausdorff topology on $X$ which makes the median $m \colon (X^3,\sigma^3) \to (X,\sigma)$ continuous. 
	Then: 
	\begin{enumerate}
		\item $\tau_{\mathrm{m}} \subseteq \sigma$.  	
		\item If, in addition, $(X,m) \in (T^m_2)$ and $\sigma$ is compact, then $\tau_{\mathrm{m}} = \sigma$. 
	\end{enumerate} 
\end{lem}
\begin{proof}
	(1) For every ordered pair \(u,v\in X\), we have
	$
	B^v_u
	=
	X\setminus S^v_u
	=
	X\setminus \phi_{u,v}^{-1}(\{u\}),
	$
	where \(\phi_{u,v}(x)=m(u,x,v)\). Since \(\sigma\) makes the median operation
	continuous, the map \(\phi_{u,v}\colon (X,\sigma)\to (X,\sigma)\) is continuous.
	As \(\sigma\) is Hausdorff, the singleton \(\{u\}\) is closed. Hence \(B^v_u\)
	is \(\sigma\)-open.
	
	Therefore all subbasic branches defining \(\tau_{\mathrm m}\) are \(\sigma\)-open. It follows that
	$
	\tau_{\mathrm m}\subseteq \sigma.
	$
	
	(2) Assume now that \(X\in (T^m_2)\) and that \(\sigma\) is compact. By
	Corollary~\ref{l:TychTop}, the space \((X,\tau_{\mathrm m})\) is Hausdorff. By part (1),
	the identity map
	$
	\operatorname{id} \colon (X,\sigma)\to (X,\tau_{\mathrm m})
	$
	is continuous. Since \((X,\sigma)\) is compact and \((X,\tau_{\mathrm m})\) is Hausdorff, this
	identity map is a homeomorphism. Hence \(\tau_{\mathrm m}=\sigma\).
\end{proof}

\begin{ex} \label{ex:LipschitzCompact} 
 There exists a compact metrizable locally convex Swiss-cheese median algebra.   
	Let
	\(
	K:=\operatorname{Lip}_1([0,1],[0,1])
	\) 
	be the space of all $1$-Lipschitz functions
	$f\colon[0,1]\to[0,1]$, equipped with the uniform topology and
	the pointwise median. Then $K$ is a compact metrizable  locally convex topological median algebra. Nevertheless, its intrinsic uniformity and intrinsic topology are trivial.
\end{ex} 
\begin{proof}
	The space $K$ is compact and metrizable in the uniform topology
	by the Arzel\`a--Ascoli theorem. It is closed under the pointwise
	median. Indeed, if $f_1,f_2,f_3\in K$, then
	\[
	\left|
	m(f_1,f_2,f_3)(s)-m(f_1,f_2,f_3)(t)
	\right|
	\leq
	\max_{1\leq i\leq3}|f_i(s)-f_i(t)|
	\leq |s-t|.
	\]
	The median operation is continuous in the uniform topology.
	
	The evaluation map
	\(
	K\to[0,1]^{[0,1]}
	\)
	is a continuous injective median homomorphism. Since $K$ is
	compact, it is a topological median embedding. Hence $K$ is
	locally convex.
	
	Let $f\neq g$ in $K$, and put
	$a=\min\{f,g\}$ and $b=\max\{f,g\}$ pointwise. Then
	\[
	[f,g]_K=[a,b]_K
	=
	\{h\in K:a(t)\leq h(t)\leq b(t)
	\text{ for every }t\in[0,1]\}.
	\]
	Choose $t_0\in(0,1)$ such that $a(t_0)<b(t_0)$, and put
	$c=(a(t_0)+b(t_0))/2$. Choose
	$0<\lambda<\min\{1,c,1-c\}$ and define
	\(
	r_+(t)=c+\lambda(t-t_0),
	\ 
	r_-(t)=c-\lambda(t-t_0).
	\) 
	Both functions belong to $K$. Put
	$h_\pm=m(a,b,r_\pm)$. Then $h_\pm\in[a,b]_K$.
	
	On some neighborhood of $t_0$ one has
	$a(t)<r_\pm(t)<b(t)$, and hence $h_\pm(t)=r_\pm(t)$.
	Choose $s<t_0<t$ in this neighborhood. Then
	$
	h_+(s)<h_-(s),
	\ 
	h_+(t)>h_-(t).
	$
	Thus $h_+$ and $h_-$ are incomparable in the natural order
	$\leq_a$ on $[a,b]_K$, which is the pointwise order.
	Consequently, $[f,g]_K$ is not a chain. 
	Since $f\neq g$ were arbitrary, $K$ has no nontrivial chain
	intervals. Therefore both $\mathcal U_m(K)$ and $\tau_{\mathrm m}(K)$
	are trivial.
\end{proof}

Thus compactness, metrizability and local convexity do not imply $(T^m_2)$.

\begin{ex}[Boolean algebras] \label{ex:boolean-dichotomy}
	Let $B$ be an infinite Boolean algebra, viewed as a median
	algebra with its standard median
	$m(x,y,z)=(x\wedge y)\vee(y\wedge z)\vee(z\wedge x)$
	(see \cite{BowditchMedian}).
	Write $A$ for the set of atoms of $B$.
	Recall that $B$ is atomic if every nonzero element lies
	above an atom, and atomless if $A=\emptyset$.
	
	\begin{enumerate}[label=(\alph*)]
		\item \textbf{General description.}
		The intrinsic uniformity $\mathcal U_m$ is the initial
		uniformity induced by the maps
		\[
		\chi_a:B\longrightarrow\{0,1\},\qquad
		\chi_a(x)=1\ \Longleftrightarrow\ a\leq x,
		\qquad a\in A,
		\]
		where $\{0,1\}$ is discrete. Consequently,
		$\tau_{\mathrm m}$ is Hausdorff if and only if $B$ is atomic,
		and is trivial if and only if $B$ is atomless.
		
		To verify the description, first note that a nontrivial
		chain interval has comparable endpoints. Indeed, if $u,v$
		are incomparable, their median interval contains the
		Boolean square $\{u\wedge v,u,v,u\vee v\}$.
		
		For $u<v$, put $a=v\wedge u^*$.
		The interval $[u,v]_B$ is median-isomorphic to the Boolean
		algebra $[0,a]$. If $a$ is not an atom, choose $0<c<a$;
		then $\{0,c,a\wedge c^*,a\}$ is a Boolean square in $[0,a]$.
		Hence $[u,v]_B$ is a chain if and only if $a$ is an atom.
		In this case, $[u,v]_B=\{u,v\}$ and
		\[
		\phi_{u,v}(x)=u\vee(x\wedge a).
		\]
		Thus its two branches are
		\[
		B^u_v=\chi_a^{-1}(\{0\}),\qquad
		B^v_u=\chi_a^{-1}(\{1\}).
		\]
		Conversely, every atom $a$ gives the chain interval $[0,a]$
		and the corresponding two-fibre partition.
		These partitions are therefore precisely the subbasic
		covers defining $\mathcal U_m$.
		
		If $B$ is atomic and $x\neq y$, then, after interchanging
		$x,y$ if necessary, $x\wedge y^*\neq0$.
		An atom $a\leq x\wedge y^*$ satisfies
		$\chi_a(x)=1$ and $\chi_a(y)=0$.
		Thus the coordinates separate points, and
		$\tau_{\mathrm m}$ is Hausdorff.
		Conversely, if $B$ is not atomic, there is $b>0$ with no
		atom below it. Then $\chi_a(b)=\chi_a(0)=0$ for every
		$a\in A$, so $b$ and $0$ have the same intrinsic
		neighborhoods. Hence $\tau_{\mathrm m}$ is not Hausdorff.
		
		Finally, if $A=\emptyset$, the initial uniformity and
		topology are trivial. If $a\in A$, then
		$\chi_a(0)=0$ and $\chi_a(a)=1$, so the two fibres of
		$\chi_a$ are nonempty proper clopen sets.
		Thus $\tau_{\mathrm m}$ is trivial precisely when
		$B$ is atomless.
		
		\item \textbf{Power set algebras.}
		For $B=\mathcal P(S)$, where $S$ is an infinite set,
		the atoms are the singletons, and
		$\chi_{\{s\}}(E)=1$ precisely when $s\in E$.
		Under the identification $\mathcal P(S)=\{0,1\}^S$,
		the intrinsic uniformity and topology are therefore the
		product uniformity and topology.
		For $S=\mathbb N$, this is the compact Cantor product topology.
		
		\item \textbf{Atomless algebras.}
		For an atomless Boolean algebra, such as the algebra of
		clopen subsets of the Cantor space or the Lebesgue measure
		algebra modulo null sets, the intrinsic uniformity and
		topology are trivial.  
	\end{enumerate}
\end{ex}  

In Proposition~\ref{p:AtomicBooleanMMC} below, we use the
atom-coordinate description to identify the MMC of an atomic
Boolean algebra and prove its uniqueness.

\begin{remark} \label{r:NotSubalgebras} 	
	The Cantor cube \(\{0,1\}^{\mathbb Z}\), considered as a median algebra, satisfies \((T^m_2)\), and its median topology is the compact product topology, by Example \ref{ex:boolean-dichotomy}(b) 
	(or, by Propositions~\ref{p:UmProduct}
	and~\ref{p:T2product} below).
	The algebra $P$ of periodic bi-sequences is a (dense, in the product topology) subalgebra of $\{0,1\}^{\Z}$. Thus,  
	one more conclusion of Example \ref{ex:PERIODIC} is that 
	the property $(T^m_2)$ is not hereditary to median
	subalgebras. Moreover,
	$
	\tau_{\mathrm m}(P)
	\subsetneq
	\tau_{\mathrm m}(\{0,1\}^{\mathbb Z})|_P.
	$
	In contrast, for \textbf{convex} subalgebras  both the intrinsic topology
	and the property $(T^m_2)$ are hereditary by
	Proposition~\ref{p:MMCconvex}. 
\end{remark}

\section{The minimal median compactification}  \label{s:MMC} 

Now we show one more important property of $\mathcal{U}_{\mathrm{m}}$ that it is compatible with the median operator. 

\begin{thm} \label{t:UniformMedian}  
	 For every median algebra $X$ 
	the median operation $m \colon X^3 \to X$ is uniformly continuous with respect to the m-uniformity $\mathcal{U}_{\mathrm{m}}$. Hence,  $(X,\tau_{\mathrm{m}})$ is a topological median algebra. 
\end{thm} 
\begin{proof}
	$\mathcal{U}_{\mathrm{m}}$ is generated by the subbase of covers $\mathcal{B}_{u,v} = \{H_1, H_2\}$, where $H_1 = B^u_v$, $H_2 = B^v_u$, and $[u,v]$ is a chain interval. So, it suffices to show that the product cover $\mathcal{B}_{u,v}^3$ on $X^3$ refines $m^{-1}(\mathcal{B}_{u,v})$.
	
	The cover $\mathcal{B}_{u,v}^3$ consists of exactly 8 sets of the form $A \times B \times C$, where $A, B, C \in \{H_1, H_2\}$. Let $A \times B \times C$ be one such set. 
	By the pigeonhole principle, two of \(A,B,C\) are equal. Since the median
	operation is symmetric, we may assume without loss of generality that \(A=B=H_i\) for some $i \in \{1,2\}$. 
	
	For any point $(x,y,z) \in A \times B \times C$, we have $x \in H_i$ and $y \in H_i$. By Lemma \ref{l:ConvexFibers}, $H_i$ is a convex halfspace.  
	Since $H_i$ is convex and $x,y \in H_i$, one has $[x,y] \subseteq H_i$; since $m(x,y,z) \in [x,y]$,  it follows that $m(x,y,z) \in H_i$. 
	This proves that $m(A \times B \times C) \subseteq H_i$. Consequently, every set in the product cover $\mathcal{B}_{u,v}^3$ is mapped entirely into a single set of $\mathcal{B}_{u,v}$. Thus, $\mathcal{B}_{u,v}^3 \succ m^{-1}(\mathcal{B}_{u,v})$, establishing that $m$ is uniformly continuous.
\end{proof}

Thus, $\mathcal{U}_{\mathrm{m}}$ is median-compatible in the sense of Definition \ref{d:median-compatible} for every Hausdorff $\mathcal{U}_{\mathrm{m}}$. 

\begin{thm} \label{t:MMC} 
	Let $X \in (T^m_2)$. 
	Denote by $\nu_m \colon X \hookrightarrow \widehat{X}$ the proper compactification induced by the (precompact) Hausdorff median-uniformity $\mathcal{U}_{\mathrm{m}}$. Then 
	\begin{enumerate}
		\item 	
		$\widehat{X}$ admits a continuous median $\widehat{m}$ which extends the median on $X$ and  
		$\widehat{X}$ is a compact, Hausdorff median algebra. 
		
		\item \(\widehat X\) is the Minimal Median Compactification of \((X,\tau_{\mathrm m})\): if
		\(K\) is any proper median compactification of \((X,\tau_{\mathrm m})\), then there exists
		a unique continuous surjective median homomorphism
		$
		f \colon K\to \widehat X
		$
		such that \(f(x)=x\) for every \(x\in X\). 
	\end{enumerate} 
\end{thm}    
\begin{proof} (1) 
$\mathcal{U}_{\mathrm{m}}$ is precompact by definition and Hausdorff because $X \in (T^m_2)$. Therefore, 
its completion $\widehat{X}$ is compact Hausdorff. 

By Theorem~\ref{t:UniformMedian}, the median operation is uniformly
continuous. Since Hausdorff completion commutes with finite
products, it extends uniquely to a continuous ternary operation
$\widehat m\colon\widehat X^3\to\widehat X$. The median identities pass to the completion by continuity and density. This proves that $\widehat{X}$ is a compact Hausdorff median algebra. 
	
	(2) To prove minimality, let \(K\) be a proper median compactification of
	\((X,\tau_{\mathrm m})\). Thus \(K\) is a compact Hausdorff topological median algebra
	and \(X\) is identified with a dense median subalgebra of \(K\).
 Because $K$ is a compact Hausdorff space, it admits a unique compatible uniformity, which induces a subspace uniformity $\mathcal{U}_K$ on $X$. By Proposition \ref{p:Median-compactif} we have to show that $\mathcal{U}_{\mathrm{m}} \subseteq \mathcal{U}_K$. 
 It is enough to show $\mathcal{B}_{u,v} \in \mathcal{U}_K$, where $\mathcal{B}_{u,v} = \{B^u_v, B^v_u\}$ is a subbasic cover of $\mathcal{U}_{\mathrm{m}}$ generated by a chain $[u,v] \subseteq X$ with $u \neq v$.  
 
	In the compact median algebra $K$, the interval $[u,v]_K$ is the topological closure of $[u,v]_X$ by Lemma \ref{l:IntervalClosure}. Moreover,  $[u,v]_K$ remains a chain in $K$ 
	by the same lemma.  
	Consider the canonical retraction $\phi^K_{u,v} \colon K \to [u,v]_K$. 
	Because $K$ is a topological median algebra, this map is continuous.  
	Since \(K\) is Hausdorff, the singletons \(\{u\}\) and \(\{v\}\)
	are closed in the interval \([u,v]_K\). Hence 	
	\[
	[u,v)_K=[u,v]_K \setminus \{v\},
	\qquad
	(u,v]_K=[u,v]_K \setminus \{u\}
	\]
	are open in the subspace topology of \([u,v]_K\).  
	 Therefore the following subsets are open in $K$:  
	$$\widetilde B^u_v=(\phi^K_{u,v})^{-1}([u,v)_K),
	\qquad
	\widetilde B^v_u=(\phi^K_{u,v})^{-1}((u,v]_K).$$	 
Since $u \neq v$, $\{\widetilde{B}_v^u, \widetilde{B}_u^v\}$ is an open cover of $K$.
	
	In a compact Hausdorff space, every open cover is a uniform cover. Therefore, $\{\widetilde{B}_v^u, \widetilde{B}_u^v\}$ is a uniform cover of $K$.  
	Its restriction to \(X\) is exactly
	\(\mathcal B_{u,v}=\{B^u_v,B^v_u\}\). Indeed,
	\[
	\widetilde B^u_v\cap X
	=
	(\phi^K_{u,v}|_X)^{-1}
	\bigl([u,v)_K\cap[u,v]_X\bigr)
	=
	\phi_{u,v}^{-1}([u,v)_X)
	=
	B^u_v,
	\]
	and similarly,
	$\widetilde B^v_u\cap X
	=
	\phi_{u,v}^{-1}((u,v]_X)
	=
	B^v_u.$ 
	
	Consequently, every subbasic cover of $\mathcal{U}_{\mathrm{m}}$ belongs to $\mathcal{U}_K$, which proves $\mathcal{U}_{\mathrm{m}} \subseteq \mathcal{U}_K$.
	  
	The identity map $\operatorname{id} \colon (X, \mathcal{U}_K) \to (X, \mathcal{U}_{\mathrm{m}})$ is uniformly continuous. By the universal property of completions, this map extends uniquely to a continuous surjection $f \colon K \to \widehat{X}$. Finally, because $f$ is the identity on the dense median subalgebra $X$, the continuity of the median operations ensures $f$ is a median homomorphism.
\end{proof}

\begin{defin} \label{d:MMC} Let $X \in (T^m_2)$. 
We call the proper compactification
$\nu_m\colon X\hookrightarrow\widehat X$ from
Theorem~\ref{t:MMC} the \textbf{Minimal Median Compactification}
(MMC) of $X$.	 

Identifying \(X\) with its image, it is natural to call the complement \(\widehat{X}\setminus X\) the MMC-boundary, or median boundary, and to denote it by \(\partial_m(X)\).   	
\end{defin}

An important property of MMC is that its compact
topology is intrinsic:  
by Proposition~\ref{p:MMCisMedHausd} below,
$\widehat X\in(T^m_2)$ and $\tau_{\mathrm m}(\widehat X)$ coincides with
the compact topology of $\widehat X$ for every $X \in (T^m_2)$.

\begin{prop} \label{p:LOTS_case}
	Let $(X,m)$ be a median algebra.
	\begin{enumerate}
		\item \textbf{The LOTS Case:} If $X$ itself is a chain equipped with a linear order $\le$, then 
		\begin{enumerate} 
			\item Median uniformity $\mathcal{U}_{\mathrm m}(X)$ coincides with the interval uniformity $\mu_\le$ from Fact \ref{f:DedekindUniformityInternal}.  
			\item Its intrinsic m-topology $\tau_{\mathrm m}(X)$ coincides with the classical interval topology $\lambda_{\leq}$.
			\item Minimal Median Compactification $\widehat{X}$ is the classical Dedekind compactification.
		\end{enumerate}
		 
		\item \textbf{Chain Intervals:} Let $C = [u, v]$ be a chain interval in an arbitrary median algebra $X$, equipped with its natural order $\le_u$. Then the retraction $\phi_{u,v} \colon (X, \Ucal_{\mathrm m}(X)) \to (C, \mu_{\leq}(C))$ is uniformly continuous. Hence, $\phi_{u,v} \colon (X, \tau_{\mathrm m}(X)) \to (C, \tau_{\mathrm m}(C))$ is  continuous. 	  
	\end{enumerate}
\end{prop} 
\begin{proof}
	(1) Suppose first that \(X\) itself is a chain and that the median is induced
	by the order \(\leq\). For \(x<y\), we have
	\(
	B^x_y=(-\infty,y),\qquad B^y_x=(x,\infty).
	\)
	Thus the subbasic covers \(\mathcal B_{x,y}=\{B^x_y,B^y_x\}\) are exactly the
	two-ray covers \(\{L_y,R_x\}\) from Lemma~\ref{l:LOTSSubbase}. Hence
	\(\mathcal U_{\mathrm m}(X)=\mu_\leq\). The equality of the induced topologies follows,
	and the compactification is the classical Dedekind compactification.
	
	(2) Let \(C=[u,v]\) be a chain interval in \(X\), ordered by \(\leq_u\).
	By Lemma~\ref{l:LOTSSubbase}, the interval uniformity \(\mu_\leq(C)\) has a
	uniform subbase consisting of the covers
	$
	O_{a,b}=\{L_b,R_a\},
	\ a<_u b,\ a,b\in C,
	$
	where
	\[
	L_b=\{t\in C:t<_u b\},\qquad R_a=\{t\in C:a<_u t\}.
	\]
	We show that the inverse image of each \(O_{a,b}\) under \(\phi_{u,v}\) belongs
	to \(\mathcal U_{\mathrm m}(X)\). 
	Since \(a,b\in C=[u,v]\), the interval \([a,b]\) is contained in \(C\), hence
	is a chain interval in \(X\). 
	
	For $z\in X$, put $p=\phi_{u,v}(z)$. Since
	$[a,b]\subseteq[u,v]$, Lemma~\ref{l:compositionOFgates}(1) gives
	$\phi_{a,b}(z)=\phi_{a,b}(p)$. On the chain $C$, one has
	$\phi_{a,b}(p)\neq b$ if and only if $p<_u b$, and
	$\phi_{a,b}(p)\neq a$ if and only if $a<_u p$.
	Consequently,
	\[
	\phi_{u,v}^{-1}(L_b)=B^a_b,
	\qquad
	\phi_{u,v}^{-1}(R_a)=B^b_a.
	\] 
	Therefore, 
	\( 
	\phi_{u,v}^{-1}(O_{a,b})=\{B^a_b,B^b_a\}=\mathcal B_{a,b},
	\) 
	which is a subbasic cover of \(\mathcal U_{\mathrm m}(X)\). Hence \(\phi_{u,v}\) is
	uniformly continuous. 
\end{proof}

\begin{prop} \label{p:ConvexInCompactification}
	Let $\eta\colon X\hookrightarrow K$ be a proper median
	compactification of a topological median algebra
	$(X,m,\sigma)$. Suppose that every interval $[u,v]_X$ is compact in the topology
	$\sigma$. Then, identifying $X$ with $\eta(X)$,
	the algebra $X$ is a convex subalgebra of $K$.
\end{prop} 
\begin{proof}
	Let $u,v\in X$. By Lemma~\ref{l:IntervalClosure},
	the median interval $[u,v]_K$ is the closure of $[u,v]_X$ in $K$.
	Since $[u,v]_X$ is compact and $K$ is Hausdorff, it is closed
	in $K$. Hence $[u,v]_K=[u,v]_X\subseteq X$. Thus $X$ is
	convex in $K$.
\end{proof}  

\begin{prop} \label{p:MMCconvex}  
	Let \(X\) be a median algebra and let \(C\subseteq X\) be convex. Then 
	\[
	\mathcal U_{\mathrm m}(C)=\mathcal U_{\mathrm m}(X)|_C
	\quad\text{and}\quad
	\tau_{\mathrm m}(C)=\tau_{\mathrm m}(X)|_C.
	\]
	In particular, if \(X\in(T^m_2)\), then \(C\in(T^m_2)\). In this case,
	the MMC of \(C\) is canonically isomorphic, as a compact topological
	median algebra, to the closure of \(C\) in the MMC of \(X\).
\end{prop} 
\begin{proof}  
	First we show that
	$
	\mathcal U_{\mathrm m}(C)\subseteq \mathcal U_{\mathrm m}(X)|_C .
	$
	Let \(\mathcal B^C_{c,d}=\{B^c_d(C),B^d_c(C)\}\) be a subbasic cover of
	\(\mathcal U_{\mathrm m}(C)\), where \([c,d]_C\) is a nontrivial chain interval in \(C\).
	Since \(C\) is convex, \([c,d]_X=[c,d]_C\). Hence \([c,d]_X\) is also a chain
	interval in \(X\), and
	$
	B^c_d(C)=B^c_d(X)\cap C,\ 
	B^d_c(C)=B^d_c(X)\cap C.
	$
	Thus \(\mathcal B^C_{c,d}\) is the restriction to \(C\) of the subbasic cover
	\(\mathcal B^X_{c,d}\).
	
	Conversely, in order to show \(
	\mathcal U_{\mathrm m}(X)|_C \subseteq \mathcal U_{\mathrm m}(C) 
	\), let
	\(
	\mathcal B^X_{x,y}=\{B^x_y(X),B^y_x(X)\}
	\) 
	be a subbasic cover of \(\mathcal U_{\mathrm m}(X)\), where
	\([x,y]_X\) is a nontrivial chain interval. We prove that its
	restriction to \(C\) belongs to \(\mathcal U_{\mathrm m}(C)\). 
	Put \(\phi=\phi_{x,y}\). Recall that
	\(
	S^x_y=\phi^{-1}(\{y\}),\ 
	S^y_x=\phi^{-1}(\{x\}).
	\) 
	If \(C\) is disjoint from one of these shadows, then one member of
	the restricted cover is equal to \(C\). Hence the restricted cover is
	a uniform cover of \(C\).
	
	Assume now that both shadows meet \(C\). Choose \(p,q\in C\) such that
	\(
	\phi(p)=y,\ \phi(q)=x.
	\) 
	Since \(C\) is convex, \([q,p]_X\subseteq C\). Define
	$
	a:=\phi_{q,p}(x),\ b:=\phi_{q,p}(y).
	$
	By Lemma~\ref{l:compositionOFgates}(2), applied to the intervals
	\([x,y]\) and \([q,p]\), the intervals \([x,y]\) and \([a,b]\)
	are naturally isomorphic. More precisely,
	\(
	\phi_{x,y}|_{[a,b]} \colon [a,b]\to [x,y]
	\) 
	is a median isomorphism satisfying
	$
	\phi_{x,y}(a)=x,\ \phi_{x,y}(b)=y.
	$
	In particular, \([a,b]\) is a nontrivial chain interval. Moreover,
	\(
	[a,b]_X\subseteq[q,p]_X\subseteq C,
	\) 
	and hence \([a,b]_C=[a,b]_X\). 
	Applying Lemma~\ref{l:compositionOFgates}(1) to the intervals
	\([a,b]\) and \([x,y]\), we obtain
	\(
	\phi_{x,y}\circ\phi_{a,b}=\phi_{x,y}.
	\) 
	Since \(\phi_{x,y}|_{[a,b]}\) is injective and maps \(a\) to \(x\)
	and \(b\) to \(y\), for every \(z\in C\) we have
	\[
	\phi_{x,y}(z)=y
	\quad\Longleftrightarrow\quad
	\phi_{a,b}(z)=b,
	\]
	and similarly,
	\[
	\phi_{x,y}(z)=x
	\quad\Longleftrightarrow\quad
	\phi_{a,b}(z)=a.
	\]
	Consequently,
	$
	B^x_y(X)\cap C=B^a_b(C),\ 
	B^y_x(X)\cap C=B^b_a(C).
	$
	Thus the restriction of \(\mathcal B^X_{x,y}\) to \(C\) is the
	subbasic intrinsic cover \(\mathcal B^C_{a,b}\). 
	Therefore
	\(
	\mathcal U_{\mathrm m}(X)|_C \subseteq \mathcal U_{\mathrm m}(C).
	\) 
	Hence \(
	\mathcal U_{\mathrm m}(X)|_C = \mathcal U_{\mathrm m}(C).
	\)    
		The equality of the induced topologies follows immediately:
	\[
	\tau_{\mathrm m}(C)=\operatorname{top}(\mathcal U_{\mathrm m}(C))
	=\operatorname{top}(\mathcal U_{\mathrm m}(X)|_C)
	=\tau_{\mathrm m}(X)|_C.
	\]
	\sk 
	If $X\in(T^m_2)$, then $\mathcal U_{\mathrm m}(X)$ is Hausdorff by
	Lemma~\ref{l:medT2Hausdorff}. Hence its restriction
	$\mathcal U_{\mathrm m}(X)|_C$ is Hausdorff, and therefore so is
	$\mathcal U_{\mathrm m}(C)$. By the same lemma, $C\in(T^m_2)$.
	
	Finally, the completion of a uniform subspace of a Hausdorff uniform
	space is canonically uniformly isomorphic to its closure in the
	completion of the ambient space. Therefore, the completion of
	\((C,\mathcal U_{\mathrm m}(C))\) is canonically isomorphic to the closure of
	\(C\) in the completion of \((X,\mathcal U_{\mathrm m}(X))\). This isomorphism
	is median-preserving, since both median operations are the unique
	continuous extensions of the median operation on the dense
	subalgebra \(C\).  
\end{proof} 

Propositions~\ref{p:MMCconvex} and
\ref{p:ConvexInCompactification} are uniform-topological analogues
of the corresponding zero-completion and Roller--Fioravanti
compactification results in
\cite[Lemma~4.8 and Theorem~4.14(2),(3)]{Fioravanti20}. 

\begin{remark} \label{r:supercomp} 
	Normal supercompactness is an important topological property in the theory of convex structures extensively studied via normal binary closed subbases. In the setting of a compact Hausdorff topological median algebra, it is equivalent to the space being locally convex. 
	This equivalence is standard in the language of binary convexities and screening
	properties; see, for example, \cite{vanMill,Vel-book,KKT}.   
\end{remark}

\begin{prop}[Intrinsic topology of the MMC]  \label{p:MMCisMedHausd} 
	Let $X \in (T^m_2)$ with its Hausdorff intrinsic median topology $\tau_{\mathrm m}(X)$ and its MMC $\nu_m \colon X \hookrightarrow \widehat{X}$. Then 
	the median algebra $\widehat{X}$ also satisfies $(T^m_2)$ and $\tau_{\mathrm m}(\widehat{X})$ is compact Hausdorff and locally convex. Moreover, $\widehat{X}$ is normally supercompact. 	
\end{prop}  
\begin{proof}
	Let \(I\) be the set of all ordered pairs \((u,v)\in X^2\) such
	that \([u,v]_X\) is a nontrivial chain interval. For every
	\(i=(u,v)\in I\), put
	\(
	C_i:=[u,v]_X
	\) 
	and denote by 
	\(
	\pi_i^X=\phi_{u,v}\colon X\to C_i
	\) 
	the canonical retraction. By 
	Theorem~\ref{t:InitialUniformity}, the uniformity
	\(\mathcal U_m(X)\) is the initial uniformity generated by the
	family of maps \(\{\pi_i^X\}_{i\in I}\), where every chain \(C_i\)
	carries its interval uniformity.
	
	Since \(C_i\) is a convex subalgebra of \(X\),
	Proposition~\ref{p:MMCconvex} gives
	\(
	\mathcal U_m(C_i)=\mathcal U_m(X)|_{C_i}.
	\) 
	By Proposition~\ref{p:LOTS_case}(1a), the intrinsic uniformity
	\(\mathcal U_m(C_i)\) is precisely the interval uniformity on the
	chain \(C_i\). Therefore, the completion of \(C_i\) is canonically
	identified with its closure in \(\widehat X\). By
	Lemma~\ref{l:IntervalClosure}, this closure is
	\(
	\overline{C_i}^{\,\widehat X}=[u,v]_{\widehat X},
	\) 
	and it is again a chain. 
	Consequently, each retraction
	\( 
	\pi_i^X \colon X\to C_i
	\) 
	extends uniquely to a continuous median map
	\(
	\pi_i^{\widehat X} \colon 
	\widehat X\to  [u,v]_{\widehat X}.
	\) 
	On the other hand, the retraction
	\(
	\phi_{u,v}^{\widehat X} \colon 
	\widehat X\to [u,v]_{\widehat X},
	\ 
	\phi_{u,v}^{\widehat X}(z)=\widehat m(u,z,v),
	\) 
	is continuous and extends \(\phi_{u,v}\) on \(X\). By uniqueness
	of the extension,
	\(
	\pi_i^{\widehat X}=\phi_{u,v}^{\widehat X}.
	\)  
	Consider the diagonal map
	\[
	\Delta_X \colon X\longrightarrow\prod_{i\in I}C_i,
	\qquad
	\Delta_X(z)=(\pi_i^X(z))_{i\in I}.
	\]
	By Theorem~\ref{t:InitialUniformity}, the uniformity induced on
	\(X\) by \(\Delta_X\) is precisely \(\mathcal U_m(X)\). Since
	\(X\in(T^m_2)\), Lemma~\ref{l:medT2Hausdorff} implies that
	\(\mathcal U_m(X)\) is Hausdorff. Hence \(\Delta_X\) is injective
	and is a uniform embedding onto its image. 
	Passing to Hausdorff completions, we obtain an embedding
	\[
	\Delta \colon \widehat X\longrightarrow
	\prod_{i=(u,v)\in I}[u,v]_{\widehat X},
	\qquad
	\Delta(z)=(\pi_i^{\widehat X}(z))_{i\in I}.
	\]
	In particular, \(\Delta\) is injective.  
	Let \(a\neq b\) in \(\widehat X\). Since \(\Delta\) is injective,
	there exists \(i=(u,v)\in I\) such that
	$
	p:=\pi_i^{\widehat X}(a)
	\neq
	q:=\pi_i^{\widehat X}(b).
	$
	The points \(p,q\) belong to the chain
	\([u,v]_{\widehat X}\). Hence
	\([p,q]_{\widehat X}\) is a nontrivial chain interval and
	$
	[p,q]_{\widehat X}\subseteq [u,v]_{\widehat X}.
	$
	By Lemma~\ref{l:compositionOFgates}(1), applied to these nested
	intervals,
	\(
	\phi_{p,q}^{\widehat X}
	\circ
	\phi_{u,v}^{\widehat X}
	=
	\phi_{p,q}^{\widehat X}.
	\) 
	Since
	\(\pi_i^{\widehat X}=\phi_{u,v}^{\widehat X}\), it follows that
	$
	\phi_{p,q}^{\widehat X}(a)
	=
	\phi_{p,q}^{\widehat X}
	\bigl(\phi_{u,v}^{\widehat X}(a)\bigr)
	=
	\phi_{p,q}^{\widehat X}(p)
	=
	p
	$
	and, similarly,
	$
	\phi_{p,q}^{\widehat X}(b)
	=
	\phi_{p,q}^{\widehat X}
	\bigl(\phi_{u,v}^{\widehat X}(b)\bigr)
	=
	\phi_{p,q}^{\widehat X}(q)
	=
	q.
	$ 
	
	Thus the chain interval \([p,q]_{\widehat X}\) separates \(a\)
	and \(b\) in the sense of \((T^m_2)\). Therefore
	\(
	\widehat X\in(T^m_2).
	\)  
	By Theorem~\ref{t:MMC}, \(\widehat X\) is a compact Hausdorff
	topological median algebra. Since
	\(\widehat X\in(T^m_2)\),
	Lemma~\ref{l:LeastToplogy} shows that
	$\tau_{\mathrm m}(\widehat X)$ coincides with the compact topology of
	$\widehat X$.
 Hence \(\tau_{\mathrm m}(\widehat X)\) is compact
	Hausdorff.
	
	By Corollary~\ref{c:BranchesSubbaseTauM}(2), the intrinsic topology
	\(\tau_{\mathrm m}(\widehat X)\) is locally convex. Therefore
	Remark~\ref{r:supercomp} implies that \(\widehat X\) is normally
	supercompact.
\end{proof}

The following proposition and its proof were suggested by the
anonymous referee.

\begin{prop} 
	\label{p:UniqueMMCCompactIntervals}
	Let $X\in(T^m_2)$, and suppose that every interval in $X$
	is compact in $\tau_{\mathrm m}$. Then the MMC is the unique
	proper median compactification of $(X,\tau_{\mathrm m})$.
	Moreover, every proper median compactification
	$X\hookrightarrow K$ satisfies $K\in(T^m_2)$, and the compact
	topology of $K$ coincides with $\tau_{\mathrm m}(K)$.
\end{prop}

\begin{proof}
	Let $X\hookrightarrow K$ be a proper median compactification,
	and identify $X$ with its image in $K$. By
	Proposition~\ref{p:ConvexInCompactification}, $X$ is convex
	in $K$. By Theorem~\ref{t:MMC}, there exists a continuous
	surjective median homomorphism
	$f:K\to\widehat X$ fixing $X$.
	
	Suppose that $a,b\in K$ and $f(a)=f(b)$.
	For every $u,v\in X$, convexity gives
	\[
	m_K(u,a,v),m_K(u,b,v)\in [u,v]_K=[u,v]_X\subseteq X.
	\]
	Since $f$ preserves medians,
	\[
	f(m_K(u,a,v))
	=\widehat m(u,f(a),v)
	=\widehat m(u,f(b),v)
	=f(m_K(u,b,v)).
	\]
	As $f$ is injective on $X$, we obtain
	$m_K(u,a,v)=m_K(u,b,v)$ for all $u,v\in X$.
	
	The continuous maps
	\[
	K^2\longrightarrow K,\qquad
	(u,v)\longmapsto m_K(u,a,v),\quad
	(u,v)\longmapsto m_K(u,b,v)
	\]
	agree on the dense subset $X^2$. Since $K$ is Hausdorff,
	they agree on all of $K^2$. Taking $u=a$ and $v=b$ gives
	\[
	a=m_K(a,a,b)=m_K(a,b,b)=b.
	\]
	Thus $f$ is injective, hence an isomorphism of compact
	topological median algebras fixing $X$.
	
	The final assertions follow from
	Proposition~\ref{p:MMCisMedHausd}.
\end{proof}

\begin{prop}[Atomic Boolean algebras]
	\label{p:AtomicBooleanMMC}
	Let $B$ be an atomic Boolean algebra, and let $A$ be its set
	of atoms. Then the MMC of $(B,\tau_{\mathrm m})$ is canonically
	isomorphic to $\{0,1\}^A$, and it is the unique proper median
	compactification of $(B,\tau_{\mathrm m})$.
\end{prop}

\begin{proof}
	The finite case is immediate, so assume that $B$ is infinite.
	The atom-coordinate description in
	Example~\ref{ex:boolean-dichotomy}(a) gives a uniform embedding
	$B\to\{0,1\}^A$.
	For every finite $F\subseteq A$, the element
	$u_F:=\bigvee F$ has coordinates $\chi_F$.
	Thus the image contains all points with finite support and
	is dense in $\{0,1\}^A$.
	Hence $\widehat B\cong\{0,1\}^A$.
	
	Let $B\hookrightarrow K$ be a proper median compactification,
	and identify $B$ with its image in $K$.
	By Theorem~\ref{t:MMC}, there is a continuous surjective median
	homomorphism $f \colon K\to\widehat B$ fixing $B$.
	Since $[0,1]_K$ is closed and contains the dense subset $B$,
	we have $K=[0,1]_K$.
	
	For every finite $F\subseteq A$, consider the continuous gate
	$p_F(z):=m_K(0,z,u_F)$.
	Its image on $B$ is the finite interval $[0,u_F]_B$.
	This interval is closed in $K$, so density and continuity give
	\[
	p_F(K)=[0,u_F]_B\subseteq B.
	\]
	
	Suppose that $f(x)=f(y)$.
	Since $f$ preserves medians, $f(p_F(x))=f(p_F(y))$.
	Both gate images belong to $B$, on which $f$ is injective;
	hence $p_F(x)=p_F(y)$ for every finite $F\subseteq A$.
	
	Directing the finite subsets of $A$ by inclusion, we have
	$u_F\to1$ in $\tau_{\mathrm m}$ and therefore in $K$.
	For every $z\in K=[0,1]_K$, continuity gives
	$p_F(z)\to m_K(0,z,1)=z$.
	Consequently,
	\[
	x=\lim_F p_F(x)=\lim_F p_F(y)=y.
	\]
	Thus $f$ is injective and hence an isomorphism, proving
	uniqueness.
\end{proof}

For example, let $B$ be the Boolean algebra of finite and
cofinite subsets of $\mathbb N$.
Its MMC is $\{0,1\}^{\mathbb N}$, and it is its unique proper
median compactification.
Nevertheless, the median interval $[0,1]_B=B$ is not compact
in $\tau_{\mathrm m}$.
Indeed, identifying subsets of $\mathbb N$ with their
characteristic functions, $(B,\tau_{\mathrm m})$ is a proper
dense subspace of the Cantor cube $\{0,1\}^{\mathbb N}$.
Thus the compactness of all intervals in $\tau_{\mathrm m}$,
assumed in Proposition~\ref{p:UniqueMMCCompactIntervals},
is not necessary for uniqueness of the proper median
compactification.

\subsection*{Generalized end-compactifications of median pretrees} 

We now specialize to rank-one median algebras and compare the MMC
with end-compactifications. As part of this comparison, we
first identify the intrinsic topology $\tau_{\mathrm m}$ with the shadow topology. 

Very informative discussions of end-compactifications can be found in the works of Bowditch \cite{B} for real trees and (in a broader context of pretrees) Malyutin \cite{Mal14}.  
Recall that a median pretree $X$ is said to be \textit{weakly complete} (see \cite{Mal14}) if every bounded arc (i.e., every nonempty convex linear subset $C \subseteq [u,v] \subseteq X$) is necessarily one of the regular intervals of the form: $[a,b], (a,b], [a,b), (a,b)$. As proved in \cite[Section 11]{Mal14}, for every weakly complete median pretree $X$, there exists a proper median compactification $e \colon X \to X^e:=X\cup \Ends(X)$, which is called the \textit{end-compactification} of $X$. This provides a broad generalization of several classical versions of end-compactifications.   

In the following result, we show that 
MMC is always well defined for every median pretree $X$. If, in addition, $X$ is weakly complete then MMC coincides with the classical end-compactification.  
   
\begin{thm}[End-compactifications of median pretrees] \label{t:MedPretreeMMC} 
	Let $X$ be a median pretree (i.e., a rank-1 median algebra), and let $\tau_s$ be its shadow topology.
	Then:
	\begin{enumerate} 
		\item 
		$\tau_{\mathrm{m}}$ is Hausdorff (i.e. \(X\in(T^m_2)\)) and MMC is well defined;  
			\item $\tau_{\mathrm{m}}=\tau_s$; 
		\item If $X$ is weakly complete, then the MMC $\nu_m \colon X \hookrightarrow \widehat{X}$ coincides with the end-compactification $e \colon X \to X^e:=X\cup \Ends(X)$ in the sense of \cite{Mal14}. 
	\end{enumerate} 
\end{thm}    
	\begin{proof}
		(1) Since \(X\) has rank \(1\), every interval \([u,v]\) is a chain. Hence, for
		\(u\neq v\),
		\[
		\{\phi_{u,v}(u),\phi_{u,v}(v)\}=\{u,v\}.
		\]
		Thus \(X\in(T^m_2)\), and \(\mathcal U_m\) is Hausdorff.
		
		(2) 
		We prove that \(\tau_{\mathrm m}=\tau_s\). First, every \(\tau_{\mathrm m}\)-basic neighborhood is
		a finite intersection of sets of the form
		$
		\operatorname{St}(x,\mathcal B_{u,v}).
		$  
		By Lemma~\ref{l:SubbasicStars}, each such star is either \(X\) or one of the 
		branches \(B^u_v,B^v_u\). Since the branches form a subbase for the shadow
		topology (see Section \ref{s:StarMedUn}), every \(\tau_{\mathrm m}\)-basic neighborhood is \(\tau_s\)-open. Hence
		\(\tau_{\mathrm m}\subseteq\tau_s\).  
		
		Conversely, let \(B^v_u\) be a shadow branch. Since \([u,v]\) is a chain,
		Proposition~\ref{p:LOTS_case}(2) shows that
		\[
		\phi_{u,v} \colon (X,\tau_{\mathrm m})\to([u,v],\tau_{\mathrm m}([u,v]))
		\]
		is continuous. By Proposition~\ref{p:LOTS_case}(1), the topology
		\(\tau_{\mathrm m}([u,v])\) is the ordinary interval topology on the chain \([u,v]\).
		Now
		\[
		B^v_u=\phi_{u,v}^{-1}([u,v]\setminus\{u\})
		=\phi_{u,v}^{-1}((u,v]),
		\]
		and \((u,v]\) is open in the interval topology of \([u,v]\). Hence \(B^v_u\)
		is \(\tau_{\mathrm m}\)-open. 
		Thus every subbasic branch of the shadow topology is
		$\tau_{\mathrm m}$-open, and therefore $\tau_s\subseteq\tau_{\mathrm m}$.
			
\sk
	(3) Assume that \(X\) is weakly complete. By
	\cite[Corollary~11.10]{Mal14}, the end-compactification
	\(e:X\to X^e=X\cup\operatorname{Ends}(X)\) is a proper compact median
	compactification of \((X,\tau_s)\). By part (2), \(\tau_s=\tau_{\mathrm m}\). Hence,
	by the minimality of the MMC, there exists a continuous surjective median
	homomorphism
	$
	q:X^e\to\widehat X
	$
	such that \(q|_X=\operatorname{id}_X\).
	
	We show that \(q\) is injective. Suppose \(u\neq v\) in \(X^e\) and
	\(q(u)=q(v)\). Since \(q\) is median preserving,
	\[
	q([u,v])\subseteq[q(u),q(v)]=\{q(u)\}.
	\]
	Thus \(q\) is constant on \([u,v]\).
	
	If \(u,v\in X\), this contradicts \(q|_X=\operatorname{id}_X\). If
	\(u\in X\) and \(v\in\operatorname{Ends}(X)\), then by
	\cite[Lemma~11.5(1)]{Mal14}, the interval \([u,v)\cap X\) is a ray in \(X\);
	hence it contains distinct points of \(X\), which \(q\) would identify. This
	again contradicts \(q|_X=\operatorname{id}_X\). Finally, if
	\(u,v\in\operatorname{Ends}(X)\), then by
	\cite[Lemma~11.5(2)]{Mal14}, the interval \((u,v)\cap X\) is a line in \(X\);
	again \(q\) would identify distinct points of \(X\), a contradiction.
	
	Therefore \(q\) is injective. Since \(X^e\) and \(\widehat X\) are compact
	Hausdorff spaces, \(q\) is a homeomorphism. 
Since $q$ is median preserving and fixes $X$, it is an isomorphism of median compactifications.	     
\end{proof}

The topological dimension of the end-compactification \(\widehat X\) in
Theorem~\ref{t:MedPretreeMMC} is at most \(1\), by Theorem~\ref{t:FiniteRankHausd}. Moreover, if \((X,\tau_{\mathrm m})\) is
separable and metrizable, then \(\widehat X\) is metrizable by Theorem~\ref{t:metrizableMMC}. 

\section{Finite rank case and algebras with finite intervals} \label{s:FiniteAndDiscrete} 

 The following definition was inspired by \cite[Lemma 2.4]{B14}. 
 
 \begin{defin}\label{d:chain-solvable}
 	A median algebra \(X\) is called \emph{chain-solvable} if for every distinct
 	pair \(x,y\in X\) there exist \(c,d\in[x,y]\) such that \(c<_x d\) and
 	\([c,d]\) is a (nontrivial) \textbf{chain} interval.
 \end{defin}
 
\begin{thm} \label{thm:chain_solvable}
	Let $X$ be a median algebra. Then $X\in(T^m_2)$ if and only if $X$ is chain-solvable.
\end{thm}
\begin{proof}
	Assume first that \(X\) is chain-solvable, and let \(x\neq y\) in \(X\).
	Choose \(c,d\in[x,y]\) such that \(c<_x d\) and \([c,d]\) is a nontrivial
	chain interval. Since \(c<_x d\) inside \([x,y]\), we have
	\(c\in[x,d]\) and \(d\in[c,y]\). Hence
	\[
	\phi_{c,d}(x)=m(c,x,d)=c,\qquad
	\phi_{c,d}(y)=m(c,y,d)=d.
	\]
	Thus the chain interval \([c,d]\) separates \(x\) and \(y\) in the sense of
	Definition~\ref{d:medT2}. Therefore \(X\in(T^m_2)\).
	  
		Conversely, assume that \(X\in(T^m_2)\), and let \(u\neq v\) in \(X\).
	There exists a nontrivial chain interval \([x,y]\) such that
	\(
	\{\phi_{x,y}(u),\phi_{x,y}(v)\}=\{x,y\}.
	\) 
	Put
	\[
	x':=\phi_{x,y}(u),\qquad y':=\phi_{x,y}(v).
	\]
	Then \(\{x',y'\}=\{x,y\}\), and hence
	\([x',y']=[x,y]\) is a nontrivial chain interval. Moreover,
	\[
	\phi_{x',y'}(u)=x',
	\qquad
	\phi_{x',y'}(v)=y'.
	\]
	
	Define
	\(
	c:=\phi_{u,v}(x'),\ 
	d:=\phi_{u,v}(y').
	\) 
	By Lemma~\ref{l:compositionOFgates}(2), the restriction of
	\(\phi_{u,v}\) to \([x',y']\) is a median isomorphism from
	\([x',y']\) onto \([c,d]\). Hence \([c,d]\) is a nontrivial
	chain interval.
	
	Furthermore, Lemma~\ref{l:compositionOFgates}(1) gives
	\(
	\phi_{c,d}
	=
	\phi_{u,v}\circ\phi_{x',y'}.
	\) 
	Therefore
	\[
	\phi_{c,d}(u)=c,
	\qquad
	\phi_{c,d}(v)=d.
	\]
	In particular, \(c\in[u,d]\), and hence \(c\leq_u d\).
	Since \(c\neq d\), we obtain \(c<_u d\). Thus \(X\) is
	chain-solvable.    
\end{proof}

\subsection*{Finite rank median algebras} 
\begin{thm} \label{t:FiniteRankHausd} 
For every finite-rank median algebra \(X\), the intrinsic uniformity
\(\mathcal U_m\) is Hausdorff. Moreover, its MMC \(\widehat X\) satisfies 
$\operatorname{rank}(\widehat X)=\operatorname{rank}(X)$ and 
$\dim \widehat X\leq \operatorname{rank}(X).$
\end{thm} 
\begin{proof}
	Let \(x\neq y\in X\). We apply Bowditch's 
	 \cite[Lemma 2.4]{B14} to the finite-rank interval
	$
	\Delta=[x,y],
	$
	with \(e^-=x\) and \(e^+=y\). Bowditch defines on \(\Delta\) the natural
	lattice order induced by the endpoints \(e^-,e^+\); in the present notation
	this is exactly the order \(\leq_x\) on \([x,y]\).
	
	Taking \(a=x\) and \(b=y\) in Bowditch's lemma, we obtain points
	\(c,d\in[x,y]\) such that \(c<_x d\) and \([c,d]\) has rank \(1\). 
	Since every rank-one median algebra is a median pretree, the interval $[c,d]$ is a (nontrivial) chain. 
	Thus \(X\) is chain-solvable. By Theorem~\ref{thm:chain_solvable},
	\(X\in(T^m_2)\), and therefore \(\mathcal U_m\) is Hausdorff. 
	
	By Proposition~\ref{p:MMCisMedHausd}, $\widehat X$ is a
	compact locally convex topological median algebra.
	Moreover, Fact~\ref{f:facts2}(1) gives
	$
	\operatorname{rank}(\widehat X)
	=
	\operatorname{rank}(X)
	<
	\infty.
	$
	Therefore \cite[Lemma~12.3.3]{BowditchMedian} implies
	\[
	\dim \widehat X
	\leq
	\operatorname{rank}(\widehat X)
	=
	\operatorname{rank}(X).
	\] 
\end{proof}

\sk 
\subsection*{Algebraically discrete median algebras} 

A median algebra $X$ is said to be (algebraically) 
 \textit{discrete} if every interval $[u,v]$ in $X$ is finite, \cite{Roller,BowditchMedian}. Let us say shortly: \textit{a-discrete}. 

\begin{thm} \label{t:DISCRETEisCV}
	For every a-discrete median algebra, the intrinsic uniformity \(\mathcal U_{\mathrm m}\) is Hausdorff. 
\end{thm}
\begin{proof} 
	Since \([x,y]\) is finite, we can choose a maximal chain in the ordered interval
	\(([x,y],\leq_x)\) from \(x\) to \(y\). Two consecutive elements \(c<_x d\)
	in this maximal chain satisfy \([c,d]=\{c,d\}\); indeed, any point of
	\([c,d]\) would lie between \(c\) and \(d\) in \(([x,y],\leq_x)\). Thus
	\([c,d]\) is a two-point chain interval. Hence \(X\) is chain-solvable, and
	Theorem~\ref{thm:chain_solvable} applies. 
\end{proof}

\begin{remark} \label{r:a-discrete examples} 
Every median graph is a-discrete as a median algebra, and hence satisfies $(T^m_2)$ by Theorem~\ref{t:DISCRETEisCV}.
Conversely, every a-discrete median algebra is the vertex set of a median graph. Equivalently, algebraically discrete median
algebras are precisely the $0$-skeleta of CAT(0) cube
complexes under the standard cubical realization; see
\cite[Theorem~10.3]{Roller}. For the foundational cubulation
and hyperplane theory, see also Sageev \cite{Sageev95}.
\end{remark}

In any median algebra $X$ the set $\mathcal{H}(X)$ of all halfspaces separate the points. The diagonal map 
$$
\iota \colon X \to \{0,1\}^{\mathcal{H}(X)}, \ x \mapsto (\chi_H(x))_{H \in \mathcal{H}(X)} 
$$
is an injective MP map. Passing to the closure $X^R:=cl(\iota(X))$ 
we get the \textit{Roller compactification} $\iota \colon X \to X^R$ equivalently describable as the 
space of ultrafilters on $\mathcal{H}(X)$, where each point $\xi \in X^R$ corresponds to a maximal collection of pairwise-intersecting halfspaces containing exactly one side of every wall. 
It has many applications.   
See \cite{Roller,Fioravanti20,BowditchMedian} for details and alternative definitions. 
 
\begin{thm} \label{thm:unique_compactification}
	Let $X$ be an a-discrete median algebra. 
	Then, for the topological median algebra \((X,\tau_{\mathrm m})\), the MMC
	\(X\to \widehat{X}\) is the unique proper median compactification, and it is
	canonically isomorphic to the classical Roller compactification \(X\to X^R\). 
\end{thm} 
\begin{proof}
By Theorem \ref{t:DISCRETEisCV} $X \in (T_2^m)$. Its intervals are finite, hence compact. Therefore Proposition \ref{p:UniqueMMCCompactIntervals} gives uniqueness of the proper median compactification. 

For the Roller identification, it remains to show that the Roller embedding induces $\tau_{\mathrm m}$. Let $H$ be a halfspace of $X$. Choose points on opposite sides of its wall, and take a maximal chain between them inside their finite interval. Two consecutive points $a,b$ cross the wall, with 
$$ 
a\in H, \ b \notin H, \ [a,b]=\{a,b\}.  
$$	
Then $H=\phi^{-1}_{a,b}(\{a\})$. 	
Indeed, if $x \in H$, then $m(a,x,b) \in [a,x] \subseteq H$. So, this median equals $a$. If $x \notin H$, a similar argument using $[b,x]$ gives $m(a,x,b)=b$.

Thus every halfspace is an intrinsic branch associated with an adjacent pair. Conversely, every intrinsic branch is a halfspace. Consequently, the topology induced by the Roller embedding is precisely $\tau_{\mathrm m}$. The Roller compactification is therefore proper, and uniqueness identifies it with the MMC.  
\end{proof}

\begin{prop} \label{prop:topologically_discrete_criterion}
	Let $X$ be an a-discrete median algebra. The intrinsic median topology $\tau_{\mathrm m}$ on $X$ is topologically discrete if and only if the median graph of $X$ is locally finite (i.e., every vertex has finite degree).
\end{prop}
\begin{proof} 
	Let $G(X)$ be the graph whose vertex set is $X$ and whose edges are the adjacent pairs.

	Assume first that \(G(X)\) is locally finite, and fix \(x\in X\). Let
	\(v_1,\ldots,v_n\) be all vertices adjacent to \(x\). For each \(i\), the branch
	\(B^x_{v_i}\) is an intrinsic open halfspace containing \(x\). We claim that
	$
	\bigcap_{i=1}^n B^x_{v_i}=\{x\}.
	$
	Indeed, if \(y\neq x\), then along any geodesic from \(x\) to \(y\) there is a first
	vertex \(v_i\) adjacent to \(x\). Thus \(v_i\in[x,y]\), equivalently
	\(m(x,y,v_i)=v_i\), and hence \(y\notin B^x_{v_i}\). Therefore \(\{x\}\) is
	\(\tau_{\mathrm m}\)-open. Since \(x\) was arbitrary, \(\tau_{\mathrm m}\) is discrete.
	
	Conversely, suppose that \(x\) has infinitely many adjacent neighbors. Let
	$
	O=H_1\cap\cdots\cap H_k
	$
	be a basic \(\tau_{\mathrm m}\)-neighborhood of \(x\), where each \(H_i\) is a
	subbasic branch containing \(x\). Each \(H_i\) is a convex halfspace. We show that
	each \(H_i\) can exclude at most one adjacent neighbor of \(x\). Indeed, if two
	distinct adjacent neighbors \(v,w\) of \(x\) belonged to \(X\setminus H_i\), then,
	since \(X\setminus H_i\) is convex and \(x\in[v,w]\), we would get
	\(x\in X\setminus H_i\), contradicting \(x\in H_i\). Hence the finite intersection
	\(O\) excludes at most finitely many adjacent neighbors of \(x\). Since \(x\) has
	infinitely many adjacent neighbors, some adjacent neighbor \(v\neq x\) still belongs
	to \(O\). Thus no basic neighborhood of \(x\) isolates \(x\), and \(\tau_{\mathrm m}\)
	is not discrete.
\end{proof}

\begin{remark} \label{rem:topology_equivalence} 
	In \cite{Roller}, Roller defines the \emph{convex topology} on an a-discrete median algebra as the topology inherited from its canonical embedding into the Roller compactification, which has the set of all halfspaces as a subbase. (This coincides precisely with the \emph{weak topology} in the sense of van de Vel \cite{Vel-book}). 	
In an a-discrete median algebra, every wall is generated by an adjacent pair;
equivalently, every halfspace is one of the two branches associated with some
adjacent pair. Hence Roller's convex topology coincides with the intrinsic
median topology \(\tau_{\mathrm m}\).
\end{remark}

\begin{ex} (The atomic space $\ell_1$ versus nonatomic $L_1([0,1])$) 
\label{ex:l1L1contrast}  
Equip $\ell_1$ and $L_1([0,1])$ with their standard
$L_1$-metrics and their canonical lattice medians
\(m(f,g,h)=(f\wedge g)\vee(g\wedge h)\vee(h\wedge f)\).
In $\ell_1$ this median is computed coordinatewise, while in
$L_1([0,1])$ the lattice operations are defined on equivalence
classes modulo equality almost everywhere.

	\begin{enumerate}
\item The median algebra $\ell_1$ satisfies $(T^m_2)$. Its intrinsic
topology $\tau_{\mathrm m}$ is the topology of coordinatewise
convergence inherited from $\mathbb R^{\mathbb N}$, strictly coarser than the metric topology.		
		
\item The median algebra $L_1([0,1])$, with respect to
		Lebesgue measure, has no nontrivial chain intervals.
		Consequently, its intrinsic uniformity and intrinsic
		topology are trivial.  
	\end{enumerate}
\end{ex} 
\begin{proof}
	(1) Let $x\neq y$ in $\ell_1$. Choose $n\in\mathbb N$ such
	that $x_n\neq y_n$, and put
	\(
	z=x+(y_n-x_n)e_n.
	\) 
	Then $[x,z]$ is a nontrivial chain interval, since only the
	$n$-th coordinate varies. 
	
	Then $x$ and $z$ differ only in the $n$-th coordinate, and
	\[
	[x,z]
	=
	\left\{
	x+(t-x_n)e_n:
	t\in[\min\{x_n,y_n\},\max\{x_n,y_n\}]
	\right\}.
	\]
	Thus $[x,z]$ is median-isomorphic to a real interval and hence is
	a nontrivial chain interval.

	 Moreover,
	\(
	\phi_{x,z}(x)=x,
	\ 
	\phi_{x,z}(y)=z.
	\) 
	Thus $[x,z]$ separates $x$ and $y$, and hence
	$\ell_1\in(T^m_2)$.   
		More generally, a nontrivial interval $[u,v]$ in $\ell_1$ is
	a chain if and only if $u$ and $v$ differ in exactly one
	coordinate. Indeed, if they differ in two coordinates, then
	$[u,v]$ contains a median copy of $\{0,1\}^2$. Consequently,
	the subbasic branches are precisely the inverse images of open
	rays under the coordinate projections, and hence $\tau_{\mathrm m}$ is
	the topology of coordinatewise convergence.

	(2) Let $f\neq g$ in $L_1([0,1])$. Choose measurable
	representatives, still denoted by $f$ and $g$, and put
	\(
	E=\{t\in[0,1]:f(t)\neq g(t)\}.
	\)
	Since $f\neq g$ in $L_1([0,1])$, the set $E$ has positive
	measure.
 By nonatomicity, there exist
	disjoint measurable subsets $A,B\subseteq E$, both of positive measure. 
	For a measurable set $C\subseteq E$, let $h_C\in L_1([0,1])$
	be the equivalence class represented by 
	\[
	h_C(t)=
	\begin{cases}
		g(t),&t\in C,\\
		f(t),&t\notin C.
	\end{cases}
	\]
	The four elements 
	\(
	f,\ h_A,\ h_B,\ h_{A\cup B}
	\) 
	belong to $[f,g]$ and form a median subalgebra isomorphic to
	$\{0,1\}^2$. Hence $[f,g]$ is not a chain. 
	Since $f\neq g$ were arbitrary, $L_1([0,1])$ has no
	nontrivial chain intervals. Therefore  both $\mathcal U_m$
	and $\tau_{\mathrm m}$ are trivial.
\end{proof}

The median metric space $\ell_1$, as a median algebra, has infinite rank, and it is not
a-discrete. So, this example is not covered by either of the two sufficient
conditions above: Theorems \ref{t:FiniteRankHausd} and \ref{t:DISCRETEisCV}.

\section{Comparison with the Roller--Fioravanti compactification}
\label{subsec:RF_comparison}

For a topological median algebra $(X, \sigma)$ with compact intervals, it is highly instructive to compare the Minimal Median Compactification (MMC) introduced in this work with the ``zero-completion'' $X \to X^{RF}$ in the sense of \cite{Fioravanti20}. We will refer to it as the \textit{Roller--Fioravanti (RF) compactification} (as in \cite{Me-TameMedian}). 
For dendrites, a related construction goes back to Ward \cite{Ward}.

RF compactification is constructed for median algebras $X$ with compact intervals. The construction is based on canonical interval retractions $\phi_{u,v} \colon X \to [u,v]_X$ for \emph{all} pairs $u, v \in X$. 

The corresponding diagonal map is an injective $\sigma$-continuous median homomorphism 
$$\nu \colon (X, \sigma) \to X^{RF} \subset \prod_{u,v} [u,v]_X.$$
 However, it is not necessarily a topological embedding (see  \cite{Fioravanti20} for details). A remarkable consequence of our framework is that this ``weaker'' inherited topology on every algebra $X \in (T_2^{m})$ is just the intrinsic topology $\tau_{\mathrm m}(X)$ by Proposition \ref{p:RFandMMC} below.

  \begin{prop}\label{p:RFandMMC}
  	Let $(X,m,\sigma)$ be a topological median algebra with compact
  	intervals such that $(X,m)\in(T^m_2)$. Denote by
  	\(
  	\nu\colon (X,\sigma)\to X^{RF}
  	\)
  	the RF compactification. Then, regarded as a compactification of
  	$(X,\tau_{\mathrm m})$, the map
  	\(
  	\nu\colon (X,\tau_{\mathrm m})\hookrightarrow X^{RF}
  	\)
  	is canonically isomorphic to the MMC. Moreover, the topology on
  	$X$ inherited from $X^{RF}$ is precisely $\tau_{\mathrm m}$ 
  	and $\tau_{\mathrm m} \subseteq \sigma$. 
  	Moreover, $X^{RF}$ is the unique proper median
  	compactification of $(X,\tau_{\mathrm m})$.
  \end{prop} 
\begin{proof}
	Let $\mathcal U_{RF}$ be the precompact uniformity on $X$
	induced by the RF compactification.
	
	For every $u,v\in X$, put $C=[u,v]$. Since $C$ is a convex
	subalgebra of $X$, Proposition~\ref{p:MMCconvex} gives
	\(
	\mathcal U_m(C)=\mathcal U_m(X)|_C.
	\)
	In particular, $C\in(T^m_2)$. By assumption, $C$ is compact in
	the topology $\sigma|_C$. Hence
	Lemma~\ref{l:LeastToplogy}(2) gives
	\(
	\tau_{\mathrm m}(C)=\sigma|_C.
	\)
	Therefore the compact uniformity on $C$ inherited from
	$\sigma$ coincides with $\mathcal U_m(C)$. 
	By Theorem~\ref{t:UniformMedian}, every canonical retraction
	$\phi_{u,v}\colon X\to C$ is uniformly continuous from
	$(X,\mathcal U_m(X))$ to
	\((C,\mathcal U_m(C))\). Thus every coordinate of the RF
	embedding is $\mathcal U_m(X)$-uniformly continuous, and hence
	\(
	\mathcal U_{RF}\subseteq\mathcal U_m(X).
	\)
	
	Conversely, the chain intervals are among the intervals used
	in the RF embedding. For a chain interval, the compact
	uniformity identified above is its interval uniformity.
	Therefore Theorem~\ref{t:InitialUniformity} gives
	\(
	\mathcal U_m(X)\subseteq\mathcal U_{RF}.
	\) 
	Thus
	\(
	\mathcal U_{RF}(X)=\mathcal U_m(X).
	\)
	By Proposition~\ref{p:Median-compactif}, the RF compactification,
	regarded as a compactification of $(X,\tau_{\mathrm m})$, is canonically
	isomorphic to the MMC. The equality of the induced topologies follows
	from the same equality of uniformities. 
	
	By Lemma~\ref{l:LeastToplogy}(1), we have
	$\tau_{\mathrm m}\subseteq\sigma$.
	Thus every interval is compact in $\tau_{\mathrm m}$.
	The uniqueness assertion now follows from
	Proposition~\ref{p:UniqueMMCCompactIntervals}. 
\end{proof}

\begin{remark} \label{r:ChatMes}  
	Chatterji and Messaci
	\cite[Theorems~3.1 and~3.18]{ChatterjiMessaci}
	show that a complete median space with the strict separation
	property has compact intervals and that their ultrafilter
	compactification, endowed with the convex topology, is
	canonically isomorphic to $X^{RF}$.
	If, moreover, $X\in(T^m_2)$, then
	Proposition~\ref{p:RFandMMC} shows that this common
	compactification is the unique proper median compactification
	of $(X,\tau_{\mathrm m})$, and hence is canonically isomorphic to $\widehat X$.
	This applies, in particular, when $X$ has finite rank.
\end{remark}

\begin{prop} 
	\label{p:rigid}
	Let $(X,m,\sigma)$ be a topological median algebra such that
	every interval in $X$ is compact in the topology $\sigma$.
	Then:
	\begin{enumerate}
		\item If
		\(
		\eta \colon (X,\sigma) \hookrightarrow K
		\)
		is a proper median compactification and
		\(K\in(T^m_2)\), 
		then $X\in(T^m_2)$, $\sigma=\tau_{\mathrm m}$,
		and $\eta$ is canonically isomorphic to the MMC of
		$(X,\tau_{\mathrm m})$. 		
		\item If, in addition, \((X,m) \in(T^m_2)\), then the
		following conditions are equivalent:
		\begin{enumerate}
			\item $(X,m,\sigma)$ admits a proper median
			compactification;
			\item $\sigma=\tau_{\mathrm m}$.
		\end{enumerate}
		If these conditions hold, the MMC is the unique proper
		median compactification of $(X,m,\sigma)$.
	\end{enumerate}
\end{prop}

\begin{proof}
	(1) By Proposition~\ref{p:ConvexInCompactification}, we may
	identify $X$ with a convex dense subalgebra of $K$. Hence
	Proposition~\ref{p:MMCconvex} gives
	$
	\mathcal U_m(X)=\mathcal U_m(K)|_X,
	\ 
	\tau_{\mathrm m}(X)=\tau_{\mathrm m}(K)|_X.
	$
	
	Since $K\in(T^m_2)$, the first equality implies 
	$X\in(T^m_2)$. Moreover, 
	Lemma~\ref{l:LeastToplogy}(2) shows that $\tau_{\mathrm m}(K)$ coincides
	with the compact topology of $K$.  
	The compactification $\eta \colon (X,\sigma) \hookrightarrow K$ is a
	topological embedding. Hence it follows that
	$
	\sigma=\tau_{\mathrm m}(K)|_X=\tau_{\mathrm m}(X).
	$
	
	The intrinsic uniformity $\mathcal U_m(K)$ is now the unique
	uniformity compatible with the compact Hausdorff topology of
	$K$. Its restriction to $X$ is $\mathcal U_m(X)$, so $K$ is
	the Hausdorff completion of $(X,\mathcal U_m(X))$. Therefore
	$\eta$ is canonically isomorphic to the MMC.

	(2) 
	Suppose first that $\sigma=\tau_{\mathrm m}$. Since 
	$X\in(T^m_2)$, its MMC is a proper median
	compactification. Thus, (b) implies (a). 
	
	\sk 
	For the converse implication first we prove the following 
	
	\textbf{Claim:} Let \(X\) be a dense convex subalgebra of a Hausdorff topological median algebra \(K\). If \(X\in(T_2^m)\), then \(K\in(T_2^m)\).
	\begin{proof}
	Fix distinct \(a,b\in K\). Choose disjoint open neighborhoods \(O_a\) and \(O_b\) of \(a\) and \(b\), respectively. Consider the continuous maps
	\[
	F,G:K^2\longrightarrow K,\qquad
	F(u,v)=m_K(u,a,v),\quad G(u,v)=m_K(u,b,v).
	\]
	The median identities give
	\(
	F(a,b)=m_K(a,a,b)=a,\ 
	G(a,b)=m_K(a,b,b)=b.
	\)
	Consequently,
	\(
	W:=F^{-1}(O_a)\cap G^{-1}(O_b)
	\) 
	is an open neighborhood of \((a,b)\) in \(K^2\). Since \(X^2\) is dense in \(K^2\), there exist \(u,v\in X\) with \((u,v)\in W\). Thus
	\[
	p:=m_K(u,a,v)\in O_a,\qquad
	q:=m_K(u,b,v)\in O_b,
	\]
	and therefore \(p\ne q\).
	
	By convexity of \(X\) in \(K\), \ 
	\(
	p,q\in C:=[u,v]_K=[u,v]_X\subseteq X.
	\) 
	The interval \(C\) is convex in \(X\), so Proposition~\ref{p:MMCconvex} implies \(C\in(T_2^m)\). Hence there exists a nontrivial chain interval \([s,t]_C\) such that
	\[
	\{\phi^C_{s,t}(p),\phi^C_{s,t}(q)\}=\{s,t\}.
	\]
	Since \(C\) is convex in \(K\), we have \([s,t]_C=[s,t]_K\); in particular, this is also a chain interval in \(K\). Moreover, the gate \(\phi^K_{s,t}\) restricts to \(\phi^C_{s,t}\) on \(C\).
	
	Now \([s,t]_K\subseteq[u,v]_K\), so the nested-gate identity from Lemma \ref{l:compositionOFgates}(1) gives
	\[
	\phi^K_{s,t}\circ\phi^K_{u,v}=\phi^K_{s,t}.
	\]
	As \(p=\phi^K_{u,v}(a)\) and \(q=\phi^K_{u,v}(b)\), it follows that
	\[
	\{\phi^K_{s,t}(a),\phi^K_{s,t}(b)\}
	=
	\{\phi^K_{s,t}(p),\phi^K_{s,t}(q)\}
	=
	\{s,t\}.
	\]
	Thus the chain interval \([s,t]_K\) separates \(a\) and \(b\) in the sense of \((T_2^m)\). Therefore \(K\in(T_2^m)\).	
	\end{proof}
	
	Now let $\eta:(X,\sigma)\hookrightarrow K$ be an arbitrary
	proper median compactification.
	By Proposition~\ref{p:ConvexInCompactification},
	$X$ is convex in $K$.
	The claim yields $K\in(T_2^m)$, so part (1) gives
	$\sigma=\tau_{\mathrm m}$ and identifies $\eta$ with the MMC.
	This proves (a)$\Rightarrow$(b) and the uniqueness assertion.
\end{proof}

    Uniqueness fails in general, even in rank one. This already reflects the classical nonuniqueness of LOTS compactifications. Let $X=\mathbb Q\cap[0,1]$ with its order median and its intrinsic
   topology, equivalently its usual order topology. 
   Its MMC is the Dedekind compactification $[0,1]$. For any irrational
   $\alpha\in(0,1)$, split the corresponding cut into two adjacent
   points $\alpha^-<\alpha^+$. The resulting compact linearly ordered
   space $K_\alpha$ contains $X$ as a dense ordered subspace and,
   equipped with its order median, is a proper median compactification
   of $X$. It is not isomorphic to the MMC, since the canonical factor
   $K_\alpha\to[0,1]$ collapses $\{\alpha^-,\alpha^+\}$ to $\alpha$.

  \sk 
   By Lemma~\ref{l:LeastToplogy}, every median metric space
   $(X,m,d)$ satisfies
   \(
   \tau_{\mathrm m} \subseteq\tau_d.
   \)
   It is therefore natural to ask when these two topologies coincide. 
   
   \begin{thm}\label{t:coincidence}
   	Let $(X,m,d)$ be a complete, connected, locally compact and
   	locally convex median metric space with compact intervals.
   	Assume that $(X,m)\in(T^m_2)$. Then:
   	\begin{enumerate} 
\item The intrinsic topology and the metric topology coincide:
$\tau_{\mathrm m}=\tau_d$.
Moreover, the RF compactification is canonically isomorphic
to the MMC of
$(X,\tau_{\mathrm m})=(X,\tau_d)$.

\item The MMC is the unique proper median compactification
of $(X,\tau_d)$.

   		\item In particular, all these conclusions hold for every
   		finite-dimensional locally compact $\mathrm{CAT}(0)$ cube
   		complex endowed with its cubical $\ell^1$-metric.
   	\end{enumerate}
   \end{thm}
   
   \begin{proof} (1) 
   	Under the standing hypotheses, the Roller--Fioravanti
   	compactification of $X$ is defined. Since $X$ is connected and
   	locally compact, Fioravanti's embedding theorem
   	\cite[Proposition~4.20]{Fioravanti20} shows that
   	$
   	\nu\colon(X,\tau_d)\hookrightarrow X^{RF}
   	$
   	is a topological median embedding.
   By Proposition~\ref{p:RFandMMC}, the
   	topology induced on $X$ from $X^{\rm RF}$ is $\tau_{\mathrm m}$, and
   	$X^{\rm RF}$ is canonically isomorphic to the MMC. Since $\nu$
   	is a topological embedding for $\tau_d$, it follows that
   	$\tau_{\mathrm m}=\tau_d$. This proves (1).

   	(2) By part (1), every interval is compact in the topology 
   	$\tau_{\mathrm m}=\tau_d$. Since $X\in(T^m_2)$,
   	Proposition~\ref{p:UniqueMMCCompactIntervals} shows that
   	the MMC is the unique proper median compactification
   	of $(X,\tau_d)$. 
   	
   	(3) 
   	Finally, a finite-dimensional locally compact CAT(0) cube
   	complex, with its cubical $\ell^1$-metric, is a locally convex
   	median metric space of finite rank. Local compactness and finite
   	dimension imply local finiteness, hence properness and
   	completeness; in particular, its intervals are compact. Thus all
   	the hypotheses of parts~(1) and~(2) hold; see also
   	\cite[Sections~2.1 and 4.2]{Fioravanti20}. 
   \end{proof}
   
   \begin{ex}[The metric hedgehog] \label{ex:MetricHedgehog}
   	Let $X$ be the metric hedgehog formed by countably many copies
   	of $[0,1]$ whose origins are identified. Then
   	\(
   	\tau_{\mathrm m}\subsetneq\tau_d.
   	\) 
   	In fact, $\tau_{\mathrm m}$ is compact and Hausdorff. Under the canonical
   	identification
   	\(
   	X^{RF}\cong(X,\tau_{\mathrm m}),
   	\)
   	the RF compactification map is the continuous bijection
   	\(
   	\operatorname{id}\colon(X,\tau_d)\to(X,\tau_{\mathrm m}).
   	\) 
   \end{ex}
   
   \begin{proof}
   	Since $X$ is a median pretree, Theorem~\ref{t:MedPretreeMMC}
   	shows that $\tau_{\mathrm m}$ is its shadow topology. On each individual
   	arm, $\tau_{\mathrm m}$ induces the ordinary interval topology. Every
   	basic $\tau_{\mathrm m}$-neighborhood of the origin contains all but
   	finitely many arms entirely and contains an initial interval in
   	each of the remaining finitely many arms. Consequently,
   	$(X,\tau_{\mathrm m})$ is the one-point compactification of the
   	topological sum
   	\(
   	\bigsqcup_{n\in\mathbb N}(0,1]_n.
   	\) 
   	In particular, $\tau_{\mathrm m}$ is compact and Hausdorff.
   	
   	On the other hand, a metric ball of radius $0<\varepsilon<1$
   	about the origin truncates every arm and therefore contains no basic $\tau_{\mathrm m}$-neighborhood of the origin. Hence
   	$\tau_{\mathrm m}\subsetneq\tau_d$. 
   	
   	Since $(X,\tau_{\mathrm m})$ is already compact, its MMC is the identity
   	compactification. By Proposition~\ref{p:RFandMMC}, the RF
   	compactification is canonically isomorphic to this MMC. Hence, under
   	the identification $X^{RF}\cong(X,\tau_{\mathrm m})$, the RF compactification
   	map is
   	\(
   	\operatorname{id}\colon(X,\tau_d)\to(X,\tau_{\mathrm m}),
   	\)
   	which is continuous and bijective but not a topological embedding. 
   \end{proof}

   \begin{cor} \label{c:0} 
   	Let $X$ be the metric hedgehog of
   	Example~\ref{ex:MetricHedgehog}, equipped with its metric topology
   	$\tau_d$. Then $(X,m,\tau_d)$ admits no proper median
   	compactification.
   \end{cor}
   
   \begin{proof}
   	The median algebra $X$ has rank one and all its intervals are
   	compact arcs. However, $\tau_{\mathrm m}\subsetneq\tau_d$ by
   	Example~\ref{ex:MetricHedgehog}. Hence 
   	Proposition~\ref{p:rigid}(2)  applies.
   \end{proof}
   
   The stronger equivariant fact that the metric hedgehog admits a
   continuous action of a Polish group having no proper
   $G$-compactification was established in \cite{Me-88}. 

\section{Further properties of the MMC}
\label{s:examples} 

\subsection*{Comparison of the intrinsic, metric and CW topologies}

By Lemma~\ref{l:LeastToplogy}, for every median metric space
$(X,m,d)$ one has
\(
\tau_{\mathrm m}\subseteq\tau_d.
\)
Theorem~\ref{t:coincidence} gives a broad class of spaces for which
equality holds. The metric hedgehog of
Example~\ref{ex:MetricHedgehog} shows that the local compactness
assumption cannot be omitted, even for a complete, connected,
locally convex rank-one median metric space with compact intervals.

\begin{remark}[Comparison with the CW-topology]
	Let $X$ be a finite-dimensional $\mathrm{CAT}(0)$ cube complex
	endowed with its cubical $\ell^1$-metric, and let
	$\tau_{\mathrm{CW}}$ denote its CW-topology. Thus a subset
	$U\subseteq X$ is $\tau_{\mathrm{CW}}$-open if and only if
	$U\cap C$ is open in every closed cube $C$; see
	\cite[Example~12.5 and Section~17.1]{BowditchMedian}. Then
	\[
	\tau_{\mathrm m}\subseteq\tau_d\subseteq\tau_{\mathrm{CW}}.
	\]
	The first inclusion follows from
	Lemma~\ref{l:LeastToplogy}. The second follows because the
	restriction of the cubical $\ell^1$-metric to every cube
	induces its standard Euclidean topology.
	
	If $X$ is locally compact, then
	Theorem~\ref{t:coincidence} gives $\tau_{\mathrm m}=\tau_d$.
	Moreover, a locally compact finite-dimensional
	$\mathrm{CAT}(0)$ cube complex is locally finite, and hence
	$\tau_d=\tau_{\mathrm{CW}}$. Therefore
	$
	\tau_{\mathrm m}=\tau_d=\tau_{\mathrm{CW}}.
	$ 
	
	Without local compactness both inclusions may be strict. For the
	metric hedgehog of Example~\ref{ex:MetricHedgehog}, one has
	$
	\tau_{\mathrm m}\subsetneq\tau_d\subsetneq\tau_{\mathrm{CW}}.
	$ 
	Indeed, the first strict inclusion was proved above. For the
	second, choose points $x_n$ on distinct arms with
	$d(x_n,0)=1/n$. The set
	$X\setminus\{x_n:n\in\mathbb N\}$ is CW-open but contains no
	metric neighborhood of the origin.
\end{remark}

\subsection*{Metrizability of the MMC}
We begin with the linearly ordered case; its proof also serves as a prototype for Theorem~\ref{t:metrizableMMC}.

\begin{lem} \label{l:SecondCountableLOTSUniformity}
	Let \((L,\leq,\lambda_\leq)\) be a separable metrizable LOTS. Then its interval
	uniformity \(\mu_\leq\) has a countable base and its Dedekind compactification is metrizable.
\end{lem} 
\begin{proof}
Since \(L\) is a separable metrizable LOTS, it is second-countable and admits a countable base consisting of convex open sets. 	
	
	We first note
	that the set of adjacent pairs in \(L\) is at most countable. Indeed, \(L\) has
	a countable convex base. For every adjacent pair \(a<b\), choose basic convex
	neighborhoods \(U,V\) such that
	\[
	a\in U\subseteq (-\infty,b), \qquad b\in V\subseteq (a,\infty).
	\]
	The pair \((U,V)\) determines \(a,b\): since \(U\) is convex, contains \(a\), and
	is contained in \((-\infty,b)\), the point \(a\) is the maximum of \(U\); similarly,
	\(b\) is the minimum of \(V\). Hence only countably many adjacent pairs exist.
	
	Let \(D\subseteq L\) be a countable dense subset, let \(E\) be the set of all
	endpoints of adjacent pairs, and put \(A:=D\cup E\). Then \(A\) is countable.
	We claim that the two-ray covers
	\[
	\mathcal O_{a,b}=\{(-\infty,b),(a,\infty)\},\qquad a<b,\quad a,b\in A,
	\]
	form a countable uniform subbase of \(\mu_\leq\).
	
	By Lemma~\ref{l:LOTSSubbase}, it is enough to show that every two-ray cover
	\(\mathcal O_{u,v}\), \(u<v\), is refined by some \(\mathcal O_{a,b}\) with
	\(a,b\in A\). If \(u,v\) are adjacent, then \(u,v\in E\), and we take
	\(a=u,b=v\). If \((u,v)\neq\emptyset\), choose \(a\in D\cap(u,v)\). If
	\((a,v)\neq\emptyset\), choose \(b\in D\cap(a,v)\); otherwise \(a,v\) are
	adjacent, so \(v\in E\), and we take \(b=v\). In either case
	\(
	u\leq a<b\leq v,
	\) 
	and therefore   
	$
	(-\infty,b)\subseteq (-\infty,v),\ (a,\infty)\subseteq (u,\infty).
	$
	Thus \(\mathcal O_{a,b}\succ \mathcal O_{u,v}\).
	
	Hence (the precompact) uniformity $\mu_\leq$ has a countable uniform base, and therefore its
	completion is compact metrizable. By
	Fact~\ref{f:DedekindUniformityInternal}, this completion is the
	Dedekind compactification of $L$. 
\end{proof}

\begin{f} \label{f:CountableInitialSubfamily}
	Let \(X\) be a second-countable space whose topology is initial with respect to a
	family of maps \(f_i \colon X\to Y_i\), \(i\in I\). Then there exists a countable
	subfamily \(I_0\subseteq I\) such that the topology of \(X\) is already initial
	with respect to the maps \(f_i\), \(i\in I_0\).
\end{f}

\begin{proof}
	Let \(\mathcal S\) be the subbase of \(X\) consisting of all sets
	\(f_i^{-1}(O)\), where \(O\subseteq Y_i\) is open. Let \(\mathcal B\) be a
	countable base of \(X\). For every \(B\in\mathcal B\) and every \(x\in B\), choose
	a finite intersection \(S_{x,B}\) of members of \(\mathcal S\) such that
	$
	x\in S_{x,B}\subseteq B.
    $
Since $B$, as a subspace of the second-countable space $X$, is
Lindelöf, choose countably many such finite intersections covering \(B\). Doing this for all \(B\in\mathcal B\), we
	use only countably many subbasic sets, and hence only countably many maps \(f_i\).
	The topology generated by these selected maps contains every member of
	\(\mathcal B\), and therefore is the original topology.
\end{proof}

\begin{thm} \label{t:metrizableMMC}
	Let \((X,m)\) be a median algebra such that \((X,\tau_{\mathrm m})\) is separable and metrizable. Then the MMC \(\widehat X\) is metrizable. 
\end{thm}
\begin{proof}
	Since \((X,\tau_{\mathrm m})\) is metrizable, the intrinsic
	uniformity \(\mathcal U_m\) is Hausdorff, and the MMC \(\widehat X\) exists. 
	By Theorem~\ref{t:InitialUniformity}, the topology \(\tau_{\mathrm m}\) is the initial
	topology generated by all canonical retractions
	$
	\phi_{u,v} \colon X\to ([u,v],\lambda_{\leq_u}),
	$
	where \([u,v]\) runs over all nontrivial chain intervals of \(X\). Since
	\((X,\tau_{\mathrm m})\) is second-countable, Fact~\ref{f:CountableInitialSubfamily}
	yields a countable family of such retractions
	$$
	\Phi:=\{\phi_n \colon X\to C_n=[u_n,v_n], \ n\in\mathbb N\},$$ 
	which already generates the topology \(\tau_{\mathrm m}\).
	
	Each \(C_n\) is a convex subalgebra of \(X\). By 
	Proposition~\ref{p:MMCconvex}, 
	$
	\tau_{\mathrm m}(C_n)=\tau_{\mathrm m}(X)|_{C_n}.
	$ 
	By Proposition~\ref{p:LOTS_case}, this topology is just the interval topology
	\(\lambda_{\leq_{u_n}}\) on the chain \(C_n\). Since \(X\) is second-countable,
	each \(C_n\) is second-countable. Hence, by
	Lemma~\ref{l:SecondCountableLOTSUniformity}, the interval uniformity
	\(\mu_{\leq_{u_n}}\) on \(C_n\) has a countable base.

	Let \(\mathcal U'\) be the initial uniformity on \(X\) generated by the countable family $\Phi$. 
	It has a countable uniform base, because the family is countable and each
	\((C_n,\mu_{\leq_{u_n}})\) has a countable uniform base. It is precompact, because
	it is the initial uniformity induced by the diagonal map into the product of the
	precompact uniform spaces \(C_n\).

	Moreover, \(\operatorname{top}(\mathcal U')=\tau_{\mathrm m}\), because the
	maps \(\phi_n\) were chosen to generate \(\tau_{\mathrm m}\) topologically.  
	Since \(\tau_{\mathrm m}\) is Hausdorff, the uniformity \(\mathcal U'\) is Hausdorff.  
	Also, $\mathcal U'\subseteq \mathcal U_m,$
	since \(\mathcal U_m\) is the initial uniformity generated by all chain
	retractions.

	The uniformity \(\mathcal U'\) is median-compatible. Indeed, since
	\(\mathcal U'\) is initial with respect to the family of maps $\Phi$, 
	it is enough to show that \(\phi_n\circ m_X \colon X^3\to C_n\) is uniformly continuous
	for every \(n\). Since \(\phi_n\) is a median homomorphism,
	$
	\phi_n\circ m_X
	=
	m_{C_n}\circ(\phi_n,\phi_n,\phi_n).
	$
	
	The map \((\phi_n,\phi_n,\phi_n): X^3 \to C_n^3\) is uniformly continuous by the
	definition of the initial uniformity. By Proposition~\ref{p:LOTS_case} and
	Theorem~\ref{t:UniformMedian}, the median operation \(m_{C_n}\) is
	uniformly continuous with respect to \(\mu_{\leq_{u_n}}\). 
	Therefore \(\phi_n\circ m_X\) is uniformly continuous for every \(n\). Hence,
	by the defining property of initial uniformities, \(m_X\) is uniformly continuous
	with respect to \(\mathcal U'\). 
	
	Let \(K\) be the completion of \((X,\mathcal U')\). Since \(\mathcal U'\) is
	Hausdorff, precompact, and has a countable base, \(K\) is compact metrizable.
	The median operation extends continuously to \(K\), so \(K\) is a compact
	metrizable median algebra. Since \(\operatorname{top}(\mathcal U')=\tau_{\mathrm m}\), the
	natural map \(X\to K\) is a proper median compactification of \((X,\tau_{\mathrm m})\).
	
	By the minimality of the MMC, there exists a continuous median-preserving surjection $q:K\to \widehat X$
	which is the identity on \(X\). On the other hand, since
	\(\mathcal U'\subseteq \mathcal U_m\), the identity map
	$
	(X,\mathcal U_m)\to (X,\mathcal U')
	$
	is uniformly continuous, and therefore extends to a continuous surjection
	$
	r \colon \widehat X\to K
	$
	which is also the identity on \(X\). The compositions \(q\circ r\) and
	\(r\circ q\) are continuous maps which restrict to the identity on the dense copy
	of \(X\). Therefore
	$
	q\circ r=\operatorname{id}_{\widehat X}, \  r\circ q=\operatorname{id}_K.
	$
	Thus \(K\) and \(\widehat X\) are isomorphic as compact median algebras. Since
	\(K\) is metrizable, so is \(\widehat X\).
\end{proof}

\begin{prop}\label{p:MetrizableCountableMMC}
	Let \(X\) be an a-discrete median algebra, and let \(\widehat X\) be its MMC.
	Then \(\widehat X\) is metrizable if and only if \(X\) is countable.
\end{prop}

\begin{proof}
	Assume first that \(X\) is countable. Then the family of all pairs \((u,v)\in X^2\)
	is countable. Hence the family of all subbasic intrinsic covers
	$
	\mathcal B_{u,v}=\{B^u_v,B^v_u\},
	$
	where \([u,v]\) is a nontrivial chain interval, is countable. Therefore
	\(\mathcal U_m\) has a countable base, namely the finite common refinements of
	these subbasic covers. Since \(X\) is a-discrete, Theorem~\ref{t:DISCRETEisCV}
	implies that \(\mathcal U_m\) is Hausdorff. Thus \(\mathcal U_m\) is a Hausdorff
	precompact uniformity with a countable base. Consequently its completion
	\(\widehat X\) is compact metrizable.
	
	Conversely, assume that \(\widehat X\) is metrizable. Since \(X\) is a-discrete,
	Theorem~\ref{thm:unique_compactification} identifies \(\widehat X\) with the
	Roller compactification of \(X\). In particular, every halfspace of \(X\) extends
	to a clopen halfspace of \(\widehat X\). 
	A compact metrizable space has only countably many clopen subsets. 
	Therefore \(X\) has at most countably many halfspaces, and hence at most countably
	many walls.
	
	Fix $x_0\in X$. For every $x\in X$, let
	$\mathcal W(x_0\mid x)$ denote the set of all walls separating
	$x_0$ and $x$. Since \(X\) is a-discrete, the interval
	\([x_0,x]\) is finite. Hence 
	$\mathcal W(x_0\mid x)$ is finite: every wall separating
	\(x_0\) and \(x\) separates some adjacent pair in the finite interval \([x_0,x]\),
	and each adjacent pair determines a unique wall.  
	The map
	\[
	X\to[\mathcal W(X)]^{<\omega},
	\qquad
	x\mapsto\mathcal W(x_0\mid x),
	\]
	is injective: if $x\neq y$, then a wall separating $x$ and $y$
	belongs to exactly one of
	$\mathcal W(x_0\mid x)$ and $\mathcal W(x_0\mid y)$.
	Since $\mathcal W(X)$ is countable,
	$[\mathcal W(X)]^{<\omega}$ is countable. Therefore $X$ is
	countable.
\end{proof}

\subsection*{Productivity of MMC} 

\begin{prop} \label{p:UmProduct}
	Let \(\{X_i\}_{i\in I}\) be any family of median algebras, and let $X:=\prod_{i\in I} X_i$
	be their product median algebra. Then we have the product formula for the m-uniformity 
	\[
	\mathcal U_m\left(\prod_{i\in I}X_i\right)
	=
	\prod_{i\in I}\mathcal U_m(X_i).
	\]
 	where the right-hand side is the product uniformity.
\end{prop}
\begin{proof}  
	Put
	\(
	X=\prod_{i\in I}X_i
	\)
	and denote by
	\(
	p_i\colon X\to X_i
	\)
	the coordinate projections. For
	\(u=(u_i),v=(v_i)\in X\), the coordinatewise definition of the
	median gives
	$
	[u,v]_X=\prod_{i\in I}[u_i,v_i]_{X_i}.
	$
	We first characterize the chain intervals in \(X\). We claim that
	\([u,v]_X\) is a chain if and only if at most one coordinate
	interval \([u_i,v_i]_{X_i}\) is non-singleton, and every
	non-singleton coordinate interval is a chain. 
	Suppose that \([u,v]_X\) is a chain. For every \(i\in I\), the
	restriction
	$
	p_i|_{[u,v]_X}\colon [u,v]_X\to [u_i,v_i]_{X_i}
	$
	is a surjective median homomorphism. Hence, by
	Lemma~\ref{l:HomImageChain}, each interval
	\([u_i,v_i]_{X_i}\) is a chain.
	
	Assume that two distinct coordinate intervals,
	\([u_i,v_i]_{X_i}\) and \([u_j,v_j]_{X_j}\), are
	non-singleton. Choose
	\[
	a_i\neq b_i\in [u_i,v_i]_{X_i},
	\qquad
	a_j\neq b_j\in [u_j,v_j]_{X_j}.
	\]
	Fixing all remaining coordinates, the four points corresponding
	to
	$
	\{a_i,b_i\}\times\{a_j,b_j\}
	$
	form a median subalgebra of \([u,v]_X\) isomorphic to
	\(\{0,1\}^2\). This is impossible, since every median subalgebra
	of a chain is itself a chain. Therefore at most one coordinate
	interval is non-singleton.
	
	Conversely, suppose that all coordinate intervals are singletons
	except possibly \([u_j,v_j]_{X_j}\), and that this interval is a
	chain. Then the coordinate projection \(p_j\) identifies
	\([u,v]_X\) canonically with \([u_j,v_j]_{X_j}\). Hence
	\([u,v]_X\) is a chain. This proves the claim.
	
	Let
	$
	\mathcal U_{\Pi}:=\prod_{i\in I}\mathcal U_m(X_i)
	$
	be the product uniformity. We first prove
	$
	\mathcal U_{\Pi}\subseteq\mathcal U_m(X).
	$
	Fix \(j\in I\), and let
	\(
	\mathcal B_{a,b}=\{B^a_b,B^b_a\}
	\)
	be a subbasic cover of \(\mathcal U_m(X_j)\), where
	\([a,b]_{X_j}\) is a nontrivial chain interval. For every
	\(i\neq j\), choose \(c_i\in X_i\), and define \(u,v\in X\) by
	\[
	u_j=a,\qquad v_j=b,
	\qquad
	u_i=v_i=c_i\quad(i\neq j).
	\]
	By the preceding characterization, \([u,v]_X\) is a nontrivial
	chain interval.
	
	For \(x=(x_i)\in X\), one has
	$
	\phi_{u,v}(x)
	=
	\bigl(\phi_{u_i,v_i}(x_i)\bigr)_{i\in I}.
	$
	Since \(u_i=v_i\) for every \(i\neq j\), it follows that
	\[
	x\in B^u_v
	\quad\Longleftrightarrow\quad
	x_j\in B^a_b,
	\]
	and similarly
	\[
	x\in B^v_u
	\quad\Longleftrightarrow\quad
	x_j\in B^b_a.
	\]
	Consequently,
	$
	p_j^{-1}(\mathcal B_{a,b})=\mathcal B_{u,v}.
	$
	Thus every coordinate projection
	$
	p_j\colon (X,\mathcal U_m(X))
	\to  (X_j,\mathcal U_m(X_j))
	$
	is uniformly continuous. Since \(\mathcal U_{\Pi}\) is the
	weakest uniformity making all coordinate projections uniformly
	continuous, we obtain
	\(
	\mathcal U_{\Pi}\subseteq\mathcal U_m(X).
	\)
	
	For the reverse inclusion, let
	\(
	\mathcal B_{u,v}=\{B^u_v,B^v_u\}
	\)
	be a subbasic cover of \(\mathcal U_m(X)\). Then
	\([u,v]_X\) is a nontrivial chain interval. By the
	characterization above, there exists a unique \(j\in I\) such
	that \([u_j,v_j]_{X_j}\) is non-singleton; moreover, this
	coordinate interval is a chain and \(u_i=v_i\) for every
	\(i\neq j\). As above,
	$
	\mathcal B_{u,v}
	=
	p_j^{-1}\bigl(\mathcal B_{u_j,v_j}\bigr).
	$
	The right-hand side is a uniform cover for the product uniformity
	\(\mathcal U_{\Pi}\). Hence every subbasic cover of
	\(\mathcal U_m(X)\) belongs to \(\mathcal U_{\Pi}\), and therefore
	$
	\mathcal U_m(X)\subseteq\mathcal U_{\Pi}.
	$
	
	Combining the two inclusions yields
	$
	\mathcal U_m\left(\prod_{i\in I}X_i\right)
	=
	\prod_{i\in I}\mathcal U_m(X_i).
	$
\end{proof}

\begin{prop}\label{p:T2product}
	The class of median algebras satisfying \((T_2^m)\) is closed under arbitrary products.
\end{prop}

\begin{proof}

	If every \(X_i\) satisfies \((T_2^m)\), then each \(\mathcal{U}_m(X_i)\) is Hausdorff by Lemma~\ref{l:medT2Hausdorff}. By Proposition~\ref{p:UmProduct},
	$
	\mathcal{U}_m\!\left(\prod_{i\in I}X_i\right)=\prod_{i\in I} \mathcal{U}_m(X_i),
	$
	and the product of Hausdorff uniformities is Hausdorff. Therefore
	$
	\prod_{i\in I}X_i
	$
	satisfies \((T_2^m)\).
\end{proof}

\begin{cor}\label{c:MMCProduct}
	Let $\{X_i\}_{i\in I}$ be a family of median algebras
	satisfying $(T^m_2)$. Then the MMC commutes with arbitrary
	products:
	\[
	\widehat{\prod_{i\in I}X_i}
	\cong
	\prod_{i\in I}\widehat X_i
	\]
	canonically, as compact topological median algebras.
\end{cor}

\begin{proof}
	By Proposition~\ref{p:UmProduct}, the median-uniformity on the product is the product uniformity.
	Since all factors are Hausdorff, their Hausdorff completions exist, and completion preserves products
	of Hausdorff uniform spaces. Hence the Hausdorff completion of $\prod_{i\in I} X_i$ with respect to its median-uniformity is canonically $\prod_{i\in I} \widehat{X_i}.$ Now apply Theorem~\ref{t:MMC}.
\end{proof}

\begin{cor} 
	For \(X=\mathbb R^n\) equipped with the coordinatewise median, the MMC is 
$\widehat{X} \cong [-\infty,+\infty]^n.$
\end{cor}
\begin{proof}
	This follows from Proposition~\ref{p:LOTS_case}(1) and
	Corollary~\ref{c:MMCProduct}.
\end{proof}

\section{G-compactifications of median actions} 
  
As noted after Corollary~\ref{c:0}, there exists a continuous
median-preserving action of a Polish group $G$ on a Polish
space $X$ such that $X$ admits no proper $G$-compactification
\cite{Me-88,Me-b}.  
  
We show in Theorem \ref{t:Gcompactification} that, for median algebras satisfying \((T^m_2)\) and equipped with their intrinsic median topology, every (separately) continuous action of a topological group \(G\) by median automorphisms is \(G\)-compactifiable via the MMC.

Recall that a uniformity \(\mathcal U\) on a \(G\)-space \(X\) is
\(G\)-\textit{bounded} if, for every \(\alpha\in\mathcal U\), there
exists a neighborhood \(V\) of the identity such that the cover
\(\{Vx:x\in X\}\) refines \(\alpha\). If, in addition, every
\(g\)-translation is a uniform automorphism, then \(\mathcal U\) is
an \textit{equiuniformity}. By
\cite[Fact~3.5]{Me-MaxEqComp}, the Hausdorff completion of an
equiuniform \(G\)-space carries a unique jointly continuous extended
action. If \(\mathcal U\) is precompact, this completion is compact.

\begin{thm} \label{t:Gcompactification}
	Let \(X\) be a median algebra equipped with its canonical \(m\)-topology \(\tau_{\mathrm m}\).
	Let \(G\) be a topological group acting separately continuously on \(X\) by median
	automorphisms. Then:
	\begin{enumerate}
		\item The \(m\)-uniformity \(\mathcal U_m\) is \(G\)-bounded.
		\item If, in addition, \(X\in (T^m_2)\), then the action extends uniquely to a
		jointly continuous action
		$
		G\times \widehat X\to \widehat X
		$
		on the MMC. Hence \(X\to \widehat X\) is a proper median \(G\)-compactification.
		In particular, the original action \(G\times X\to X\) is jointly continuous.
	\end{enumerate}
\end{thm} 
\begin{proof}  
	By Definition \ref{d:StarMedianUniformity}, $\mathcal{U}_{\mathrm{m}}$ admits a canonical uniform subbase $\mathcal{S}$ consisting of all $2$-element covers of the form $\mathcal{B}_{u,v} = \{B^u_v, B^v_u\}$, where $u \neq v$ are points such that the interval $[u,v]$ is a chain. 
	For every \(g\in G\), since \(x\mapsto gx\) is a median automorphism, we have
	$
	g([u,v])=[gu,gv].
	$
	Hence \([u,v]\) is a chain iff \([gu,gv]\) is a chain, and
	$
	g(B^u_v)=B^{gu}_{gv}.
	$ 
	Therefore each \(g\)-translation is a uniform automorphism of \((X,\mathcal U_m)\) sending subbasic intrinsic covers to subbasic intrinsic covers.
	
	We now prove that \(\mathcal U_m\) is \(G\)-bounded.  
	Since finite refinements of \(G\)-bounded covers are again
	\(G\)-bounded, and since every uniform cover is coarser than the finite common refinement of subbasic covers, it is enough to verify \(G\)-boundedness on the
	subbasic covers \(\mathcal B_{u,v}=\{B^u_v,B^v_u\}\), where
	\([u,v]\) is a nontrivial chain interval and
	\(
	B^u_v=\{z\in X:m(u,z,v)\neq v\},\
	B^v_u=\{z\in X:m(u,z,v)\neq u\}.
	\) 
	
	We prove that \(\mathcal B_{u,v}\) is \(G\)-bounded. We must find a neighborhood $V$ of $e \in G$ such that for every $x \in X$, the orbit $Vx$ is entirely contained in $B^u_v$ or entirely contained in $B^v_u$.
	
	Consider the natural order $\leq_u$ on the chain $[u,v]$, defined
	by
	$
	x\leq_u y\Longleftrightarrow x\in[u,y].
	$
	Since \(u<_u v\), there exist open
	neighborhoods \(W_u\) of \(u\) and \(W_v\) of \(v\) in the interval topology
	of \([u,v]\) such that
	\[
	a<_u b\qquad\text{for every }a\in W_u,\ b\in W_v .
	\]
	Indeed, if \((u,v)\neq\emptyset\), choose \(c\in(u,v)\) and take
	\(W_u=[u,c)\), \(W_v=(c,v]\); if \(u,v\) are adjacent, take
	\(W_u=\{u\}\), \(W_v=\{v\}\).
	
	By Proposition~ \ref{p:LOTS_case}(2), the retraction
	$
	\phi_{u,v} \colon (X,\tau_{\mathrm m})\to([u,v],\lambda_{\leq_u})
	$
	is continuous.  
	Since the orbit maps \(g\mapsto gu\) and \(g\mapsto gv\) are
	continuous and inversion in \(G\) is continuous, the maps
	\[
	G \to [u,v], \ g\mapsto \phi_{u,v}(g^{-1}u),\qquad
	G \to [u,v], \ g\mapsto \phi_{u,v}(g^{-1}v)
	\]
	are continuous at \(e\). Hence there is a symmetric neighborhood
	\(V=V^{-1}\) of \(e\) such that 
	\[
	u_g:=\phi_{u,v}(g^{-1}u)\in W_u,\qquad
	v_g:=\phi_{u,v}(g^{-1}v)\in W_v .
	\]
	for every \(g\in V\). 
	In particular \(u_g<_u v_h\) for all \(g,h\in V\).
	
	We claim that for every given \(x\in X\), the set \(Vx\) is contained in one member
	of \(\mathcal B_{u,v}\). Suppose not. Then there exist \(g_1,g_2\in V\) such that
	\(
	g_1x\notin B^u_v,\ g_2x\notin B^v_u .
	\) 
	The first condition means \(m(u,g_1x,v)=v\). Applying \(g_1^{-1}\) gives
	\(
	m(g_1^{-1}u,x,g_1^{-1}v)=g_1^{-1}v .
	\) 
	
	Since $\phi_{u,v}\colon X\to [u,v]$ is a median homomorphism, for
	all $a,b,c\in X$ we have
	\[
	\phi_{u,v}\bigl(m(a,b,c)\bigr)
	=
	m_{[u,v]}\bigl(\phi_{u,v}(a),\phi_{u,v}(b),
	\phi_{u,v}(c)\bigr),
	\]
	where $m_{[u,v]}$ denotes the median operation restricted to
	$[u,v]$. Applying this identity to
	\[
	g_1^{-1}v=m(g_1^{-1}u,x,g_1^{-1}v)
	\]
	and using the definitions
	\(
	u_{g_1}=\phi_{u,v}(g_1^{-1}u),
	\ 
	v_{g_1}=\phi_{u,v}(g_1^{-1}v),
	\) 
	we obtain
	\[
	v_{g_1}
	=
	m_{[u,v]}\bigl(u_{g_1},\phi_{u,v}(x),v_{g_1}\bigr).
	\]
	Since $[u,v]$ is a chain and $u_{g_1}<_u v_{g_1}$, its median is
	the order median. Therefore the last equality is possible only if
	\(
	v_{g_1}\leq_u\phi_{u,v}(x).
	\) 	
	Similarly, the second condition means \(m(u,g_2x,v)=u\). Applying \(g_2^{-1}\)
	and then \(\phi_{u,v}\), we get
	\[
	u_{g_2}
	=
	m_{[u,v]}(u_{g_2},\phi_{u,v}(x),v_{g_2}),
	\]
	and since \(u_{g_2}<_u v_{g_2}\), this implies
	$
	\phi_{u,v}(x)\leq_u u_{g_2}.
	$
	Thus
	\[
	v_{g_1}\leq_u \phi_{u,v}(x)\leq_u u_{g_2}.
	\]
	But \(v_{g_1}\in W_v\) and \(u_{g_2}\in W_u\), so by the choice of
	\(W_u,W_v\) we have
	\[
	u_{g_2}<_u v_{g_1},
	\]
	a contradiction. Therefore \(Vx\) is contained in one member of
	\(\mathcal B_{u,v}\) for every \(x\in X\), and \(\mathcal B_{u,v}\) is
	\(G\)-bounded. 
	We have shown that \(\mathcal U_m\) is \(G\)-bounded.  
	At the beginning we already established that 
	each \(g\)-translation is a uniform automorphism of \((X,\mathcal U_m)\).   
	So, \(\mathcal U_m\) is an equiuniformity.
	
	If $X\in(T^m_2)$, then $\mathcal U_m$ is Hausdorff, and its
	completion $\widehat X$ is the MMC by
	Theorem~\ref{t:MMC}. Since $\mathcal U_m$ is a precompact
	equiuniformity, the completion theorem
	\cite[Fact~3.5]{Me-MaxEqComp} yields a unique jointly continuous
	extended action
	$
	G\times\widehat X \longrightarrow \widehat X.
	$
	Hence the original action on \(X\) is jointly continuous. 
	Each extended \(g\)-translation on \(\widehat X\) is a median automorphism: its inverse is the extended \(g^{-1}\)-translation, and median preservation follows from the density of \(X\) and the
	continuity of the extended median operation. 
\end{proof}

\begin{remark}\label{r:CompanionComparison}
The companion preprint \cite{Me-TameMedian} studies
	continuous actions on Roller--Fioravanti compactifications under
	the assumption that all median intervals are compact, using the
	uniformity induced by all canonical interval retractions and a
	related $G$-boundedness argument.
	
	The setting here is different. The intrinsic uniformity
	$\mathcal U_m$ is generated by retractions onto chain intervals
	alone, no compactness assumption on arbitrary intervals is
	required, and separate continuity of the action on
	$(X,\tau_{\mathrm m})$ is sufficient. The preprint
	\cite{Me-TameMedian} is used in the present paper only for the
	final conclusion that continuous actions on compact
	finite-rank median algebras are Rosenthal representable and
	dynamically tame.
\end{remark}

A Banach space $V$ is said to be {\em Rosenthal} if it does not contain an isomorphic copy of $\ell_1$, or equivalently, if $V$ does not contain a sequence which is equivalent to an $\ell_1$-sequence. 

Let $V$ be a Banach space and let $\operatorname{Iso}(V)$ be the topological group (with the strong operator topology) of all linear onto isometries $V \to V$. For every continuous homomorphism $h \colon G \to \operatorname{Iso}(V)$, we have a canonically induced dual continuous action on the weak-star compact unit ball $B_{V^*}$ of the dual space $V^*$. So, we get a $G$-space $B_{V^*}$. 

A natural question is which continuous actions 
of $G$ on a topological space $X$ can be represented as a $G$-invariant subspace of $B_{V^*}$ for a certain Banach space $V$ from a nice class of (low-complexity) spaces.   
A \(G\)-space \(X\) is said to be \textit{Rosenthal representable} if there exist a Rosenthal Banach space \(V\), a continuous homomorphism
\(
h \colon G\to \operatorname{Iso}(V),
\) 
and a \(G\)-equivariant topological embedding
\(
X\to B_{V^*},
\) 
where \(B_{V^*}\) carries its weak-star topology and the dual
\(G\)-action.

Median preserving continuous $G$-actions of a topological group $G$ on compact median pretrees 
are Rosenthal representable \cite{GM-D}.   
It is well known (see \cite{GM-tLN,GM-TC,Me-b}) 
that Rosenthal representability of compact $G$-spaces implies dynamical tameness. 
A compact $G$-system $X$ is said to be \emph{tame} (see, for example, \cite{GM-tLN,GM-TC}) if for every continuous real function $f \in C(X)$ its orbit $fG = \{f_g \colon  X \to \R \mid f_g(x) = f(gx), g \in G\}$ is combinatorially small; namely, if $fG$ does not contain any infinite  
\textit{independent} (in the sense of H.P. Rosenthal \cite{Ro}) sequence.

\begin{cor}[Canonical $G$-compactifiability and tameness]
	\label{c:CanonicalGComp}
	Let $X\in(T^m_2)$ be equipped with its intrinsic median
	topology. Then every separately continuous action of a
	topological group $G$ on $X$ by median automorphisms admits a
	proper median $G$-compactification, namely the MMC. In
	particular, the original action is jointly continuous.
	
	If $X$ has finite rank, then the compact $G$-system
	$\widehat X$ is Rosenthal representable and dynamically tame.
\end{cor}

\begin{proof}
	By Theorem \ref{t:Gcompactification}, the extended action $G \times \widehat{X} \to \widehat{X}$ is jointly continuous, making $\widehat{X}$ a proper median $G$-compactification of $X$. 
	
	Assume now that $X$ has finite rank. Since $X$ is dense in
	$\widehat X$, Fact~\ref{f:facts2}(1) gives
	\(
	\operatorname{rank}(\widehat X)
	=
	\operatorname{rank}(X)<\infty.
	\)
	Hence $\widehat X$ is a compact finite-rank median algebra.
	By \cite[Theorem~5.2]{Me-TameMedian}, the compact
	$G$-system $\widehat X$ is Rosenthal representable and
	dynamically tame.
\end{proof}

\begin{cor} \label{cor:boundary_action_continuous}
	Let $X$ be a median algebra satisfying $(T^m_2)$ equipped with its canonical median topology $\tau_{\mathrm m}$, and let $G$ be a topological group acting on $X$ by median automorphisms. If the action $G \times X \to X$ is separately   continuous, then the induced action of $G$ on the median-boundary $\partial_m(X) = \widehat{X} \setminus X$ is jointly continuous. 
	
	Furthermore, if $X$ has finite rank, the boundary $G$-space $\partial_m(X)$ is Rosenthal representable.
\end{cor}
\begin{proof}
	By Theorem~\ref{t:Gcompactification}, the action extends to a jointly continuous action
	\(G\times\widehat X\to\widehat X\). Since \(X\) is \(G\)-invariant in
	\(\widehat X\), its complement
	$
	\partial_m(X)=\widehat X\setminus X
	$
	is also \(G\)-invariant. Therefore the restricted action
	$
	G\times\partial_m(X)\to\partial_m(X)
	$
	is jointly continuous. 
	If \(X\) has finite rank, then \(\widehat X\) is a compact finite-rank median
	algebra. By \cite{Me-TameMedian}, the compact \(G\)-system \(\widehat X\) is
	Rosenthal representable. Since \(\partial_m(X)\) is a \(G\)-subspace of
	\(\widehat X\), it is Rosenthal representable as well. 
\end{proof}

\bibliographystyle{amsplain}  

\begin{thebibliography}{10}

\bibitem{B}
B.H. Bowditch, \textit{Treelike structures arising from continua and convergence groups}, 
Mem. Am. Math. Soc., \textbf{139} (1999), n. 662

\bibitem{B14}
B.H. Bowditch, \textit{Embedding median algebras in products of trees},  Geometriae Dedicata \textbf{170} (2014), 157--176

\bibitem{BowditchMedian}
B.H. Bowditch, \emph{Median Algebras}, preprint,
first draft March 2022, revised 12 May 2024, 280 pp., available on author's homepage  
\url{https://bhbowditch.com/papers/median-algebras.pdf} 

\bibitem{CDH} 
I. Chatterji, C. Drutu, F. Haglund, \textit{Kazhdan and Haagerup properties from the median
	viewpoint}, Adv. Math. \textbf{225} (2010) 882--921

\bibitem{ChatterjiMessaci}
I. Chatterji and L. Messaci,
\emph{Median spaces with the strict separation property
	and superrigidity of actions},
preprint (2026), arXiv:2610.27866  

\bibitem{CMT} 
M. Couceiro, J.L. Marichal, B. Teheux  
\textit{Conservative Median Algebras and Semilattices}, Order, \textbf{33} (2016), 121--132


\bibitem{Eng89}
R. Engelking, {\em  General Topology\/} 
(revised and completed edition), 
Heldermann Verlag, Berlin, 1989 

\bibitem{Fioravanti20}
E. Fioravanti, \emph{Roller boundaries for median spaces and algebras}, Algebraic and Geometric Topology \textbf{20} (2020), 1325--1370

	
	\bibitem{GM-tLN}
	E. Glasner and M. Megrelishvili,
	\emph{More on tame dynamical systems}, in: Lecture Notes in Math. v. 2213, Ergodic Theory and Dynamical Systems in their Interactions with Arithmetics and Combinatorics, Eds.: S. Ferenczi, J. Kulaga-Przymus, M. Lemanczyk,  Springer, 2018, pp. 351--392  

	
	\bibitem{GM-D}
	E. Glasner and M. Megrelishvili, 
	\emph{Group actions on treelike compact spaces}, Science China Math., 
	\textbf{62} (2019), n. 12, 2447--2462   
	

	\bibitem{GM-TC} 
	E. Glasner and M. Megrelishvili, 
	\textit{Todor\u{c}evi\'{c}' Trichotomy and a hierarchy in the class of tame dynamical systems}, Trans. Amer. Math. Soc. \textbf{375} (2022), 4513--4548  
		
\bibitem{Isb}
J.R. Isbell,
\emph{Uniform Spaces},
Mathematical Surveys, v.~12,
Amer. Math. Soc., Providence, RI, 1964 
	
	\bibitem{KKT} 
	W. Kubis, A. Kucharski and S. Turek, 
	\textit{Parovicenko spaces with structures}, 
	RACSAM \textbf{108} (2014), 989--1004
				
	\bibitem{Mal14} 
	A.V. Malyutin, \textit{Pretrees and the shadow topology}, 
	St. Petersburg Math. J., \textbf{26} (2015), 225--271	
	
	
	\bibitem{Me-88}
	M.~Megrelishvili,
	{\it A Tychonoff $G$-space not admitting a compact Hausdorff $G$-extension or a $G$-linearization},
	Russian Math. Surveys {\bf 43}:2 (1988), 177--178
	
	
	\bibitem{Me-MaxEqComp}  
	M. Megrelishvili, 
	\textit{Maximal equivariant compactifications}, Topology Appl. \textbf{329} (2023), 108372 
	
	\bibitem{Me-b}
	M. Megrelishvili, 
	\textit{Topological Group Actions and Banach Representations}, 2026. Unpublished book, available on author's homepage  \url{https://u.math.biu.ac.il/~megereli/B.pdf} 
	
	\bibitem{Me-CircTop}  
	M. Megrelishvili, 
	\textit{Circular orders: Topology and continuous actions}, arXiv:2512.17314v3, 2026 
	
	\bibitem{Me-TameMedian}
	M. Megrelishvili,
	\emph{Tameness of actions on finite rank median algebras}, 
	arXiv:2601.01681v5, 2026 

\bibitem{vanMill} 
J. van Mill, \textit{Supercompactness and Wallman Spaces}, Math. Centre Tracts \textbf{85}, Amsterdam 1977	
	

	\bibitem{Roller} 
	M.A. Roller, \textit{Poc sets, median algebras and group actions: an extended study of Dunwoody’s construction and Sageev’s theorem}, Habilitationsschrift, Universität
	Regensburg (1998) arXiv:1607.07747   
	
	\bibitem{Ro} H.P. Rosenthal, 
	\emph{A characterization of Banach spaces containing $l^1$}, Proc. Nat. Acad. Sci. U.S.A., \textbf{71} (1974), 2411--2413
	
\bibitem{Sageev95} 	
M. Sageev, \textit{Ends of group pairs and non-positively curved cube complexes}, Proc. London
Math. Soc. \textbf{71} (1995) 585--617	
		
\bibitem{Sholander54}
M. Sholander, \emph{Medians and betweenness}, Proc. Amer. Math. Soc. \textbf{5} (1954), 801--807		
		
\bibitem{Vel-book}   
M.L.J. van de Vel, \emph{Theory of Convex Structures}, North-Holland Mathematical Library, Vol. 50, Elsevier, 1993 

\bibitem{Ward} 
L. E. Ward, On dendritic sets, Duke Math. J. 25  (1958), 505--514 
		
\end{thebibliography}
     
\end{document}